\documentclass[a4paper]{amsart}

\usepackage{amsmath}
\usepackage{amssymb}
\usepackage{amsthm}
\usepackage{amsfonts}
\usepackage{amstext}
\usepackage{amsopn}
\usepackage{amsxtra}
\usepackage{mathrsfs}
\usepackage{color}      
\usepackage{verbatim}   

\usepackage{pdfsync}

\newtheorem{lemma}{Lemma}[section]
\newtheorem{theorem}{Theorem}
\newtheorem{proposition}[lemma]{Proposition}

\theoremstyle{definition}
\newtheorem{definition}[lemma]{Definition}
\theoremstyle{definition}
\newtheorem{remark}[lemma]{Remark}
\theoremstyle{definition}

{\catcode `\@=11 \global\let\AddToReset=\@addtoreset}
\AddToReset{equation}{section}

\definecolor{purple}{rgb}{0.65, 0, 1}
\definecolor{orange}{rgb}{1,.5,0}
\definecolor{brown}{rgb}{.9,.73,.26}
\definecolor{dgreen}{rgb}{0,0.8,0.2}

\newcommand{\E}{\mathrm{e}}
\newcommand{\I}{\mathrm{i}}
\newcommand{\D}{\mathrm{d}}
\newcommand{\h}{\mathcal{H}}
\newcommand{\C}{\mathbb{C}}
\newcommand{\R}{\mathbb{R}}
\newcommand{\N}{\mathbb{N}}
\newcommand{\diver}{\mathrm{div}}
\newcommand{\e}{{\varepsilon }}

\newcommand{\supp}{{\mathrm{supp}}}

\def\tr{\mathop{\rm tr}\nolimits} 

\begin{document}

\title[Stationary states and large time behavior of WFP] {The
 Wigner-Fokker-Planck equation: Stationary states and large time
 behavior}

\thanks{A.~Arnold acknowledges partial support from the FWF (project ``Quantum Transport Equations:
Kinetic, Relativistic, and Diffusive Phenomena'' and Wissenschafts\-kolleg ``Differentialgleichungen''), the \"OAD (Amadeus project), and the Newton Institute of Cambridge University.
 I.~M. Gamba is supported by NSF-DMS 0807712.  M. P. Gualdani is
 supported by NSF-DMS-1109682. C. Mouhot would like to thank
 Cambridge University who provided repeated hospitality in 2009
 thanks to the Award No. KUK-I1-007-43, funded by the King Abdullah
 University of Science and Technology (KAUST). C. Sparber has been
 supported by the Royal Society through his Royal Society University
 Research Fellowship. Support from the Institute of Computational
Engineering and Sciences at the University of Texas at Austin is also
gratefully acknowledged}

\author[A. Arnold]{Anton Arnold} \address{Institute for Analysis and
 Scientific Computing, Technical University Vienna, Wiedner
 Hauptstr. 8, A-1040 Vienna, AUSTRIA}
\email{anton.arnold@tuwien.ac.at}

\author[I. Gamba] {Irene M. Gamba} \address{Department of
 Mathematics, The University of Texas at Austin, 1 University Station
 C1200, Texas 78712, USA} \email{gamba@math.utexas.edu}

\author[M. P. Gualdani]{Maria Pia Gualdani} \address{Department of
 Mathematics, The University of Texas at Austin, 1 University Station
 C1200, Texas 78712, USA} \email{gualdani@math.utexas.edu}

\author[S. Mischler]{St\'ephane Mischler}
\address{CEREMADE, Universit\'e Paris-Dauphine, Place du Marechal de
 Lattre de Tassigny F-75775 Paris Cedex 16, FRANCE}
\email{mischler@ceremade.dauphine.fr}

\author[C. Mouhot]{Cl\'ement Mouhot} \address{DPMMS, Centre for
 Mathematical Sciences, Cambridge University, Wilberforce Road,
 Cambridge CB3 0WA, UK \\
 {\it On leave from}: DMA, \'ENS, 45 rue d'Ulm F-75230 Paris cedex
 05, FRANCE} \email{Clement.Mouhot@ens.fr}

\author[C.~Sparber]{Christof Sparber}
\address{Department of Mathematics, Statistics, and Computer Science, M/C 249, University of Illinois at Chicago, 851 S. Morgan Street, Chicago, IL 60607, USA}
\email{sparber@math.uic.edu}


\begin{abstract}
 We consider the linear Wigner-Fokker-Planck equation subject to confining
 potentials which are smooth perturbations of the harmonic oscillator
 potential. For a certain class of perturbations we prove that
 the equation admits a unique stationary solution in a weighted
 Sobolev space. A key ingredient of the proof is a new result on the
 existence of spectral gaps for Fokker-Planck type operators in certain weighted $L^2$--spaces.  In
 addition we show that the steady state corresponds to a positive
 density matrix operator with unit trace and that the solutions of
 the time-dependent problem converge towards the steady state with an
 exponential rate. \end{abstract}

\maketitle

\section{Introduction}\label{sint}

This work is devoted to the study of the \emph{Wigner-Fokker-Planck
 equation} (WFP), considered in the following dimensionless form
(where all physical constants are normalized to one for simplicity):
\begin{equation}\label{qfp}
\left \{
\begin{aligned}
 \partial_t w + \xi \cdot\nabla _xw +\Theta [V] w =
 & \,  \Delta _\xi w +2  \, \textrm{div}_\xi \, (\xi w ) +  \Delta _x w ,\\
 w \big |_{t=0} = & \ w _0(x,\xi),
\end{aligned}
\right.
\end{equation}
where $x, \xi \in \R^d$, for $d \geq 1$, and $t\in \R_+$.
Here, $w(t,x,\xi)$ is the (real valued) \emph{Wigner transform} \cite{Wi} of a quantum mechanical
\emph{density matrix} $\rho (t,x, y)$, as defined by
\begin{equation}\label{trans}
w (t, x,\xi ):=\frac{1}{(2\pi )^d}\int_{\mathbb R^d}\rho \left(t, x+\frac{\eta }{2}\ , \ x-\frac{\eta }{2} \right) \E^{-\I \xi \cdot \eta} \, \D \eta.
\end{equation}
Recall that, for any time $t \in \R_+$, a quantum mechanical
(mixed) state is given by a positive, self-adjoint \emph{trace class operator}
$\rho(t) \in \mathscr T^+_1$.
Here we denote by $\mathscr B(L^2(\R^d))$ the set of bounded operators on $L^2(\R^d)$ and
by
$$
\mathscr T_1:=\{ \rho \in \mathscr B(L^2(\R^d)): \tr |\rho| < \infty \} \, ,
$$
the corresponding set of trace-class operators. We consequently write $\rho \in \mathscr T^+_1 \subset \mathscr T_1$, if in addition $\rho \geq 0$ (in the sense of non-negative operators).
Since $\mathscr T_1 \subset \mathscr T_2$, the space of \emph{Hilbert-Schmidt operators}, i.e.
$$
\mathscr T_2:=\{ \rho \in \mathscr B(L^2(\R^d)): \tr (\rho^* \rho) < \infty \} \, ,
$$
we can identify the operator $\rho(t)$ with its corresponding integral
kernel $\rho(t,\cdot, \cdot)\in L^2 (\R^{2d})$, the so-called
\emph{density matrix}.  Consequently, $\rho(t)$ acts on any given
function $\varphi \in L^2(\R^d)$ {\it via}
\begin{equation*}
(\rho(t)\varphi)(x)=\int_{\mathbb R^d} \rho (t, x,y ) \, \varphi(y) \, \D y.
\end{equation*}
Using the Wigner transformation \eqref{trans}, which by definition
yields a \emph{real-valued} function $w(t, \cdot, \cdot) \in
L^2(\R^{2d})$, one obtains a \emph{phase-space description} of quantum
mechanics, reminiscent of classical statistical mechanics, with $x\in
\R^d$ being the position and $\xi\in \R^d$ the momentum.  However, in
contrast to classical phase space distributions, $w(t, x, \xi)$ in
general also takes \emph{negative values}.

Equation \eqref{qfp} governs the time evolution of $w(t,x,\xi)$ in the
framework of so-called \emph{open quantum systems}, which model both
the Hamiltonian evolution of a quantum system and its interaction with
an environment (see \cite{Da}, e.g.). Here, we specifically describe
these interactions by the Fokker-Planck (FP) type diffusion operator
on the r.h.s. of \eqref{qfp}.  For notational simplicity we use here
only normalized constants in the quantum FP operator.  However, all of
the subsequent analysis also applies to the general WFP model
presented in \cite{SCDM} (cf. Remark \ref{rem1} below).  Potential
forces acting on $w(t,\cdot,\cdot)$ are taken into account by the
pseudo-differential operator
\begin{align}\label{theta}
 (\Theta [V] f)(x,\xi) := -\frac{\I}{(2\pi )^d }\iint_{\mathbb
   R^{2d}} \delta V (x,\eta) \, f (x,\xi')\ \E^{\I \eta \cdot (\xi
   -\xi')} \, \D \xi' \, \D \eta ,
\end{align}
where the symbol $\delta V$ is given by
\begin{align}\label{delta}
 \delta V (x,\eta)= V\left({ x+\frac{\eta}{2}}\right)-V\left({
     x-\frac{\eta }{2}}\right),
\end{align}
and $V$ is a given real valued function.
The WFP equation is a kinetic model for quantum mechanical
charge-transport, including diffusive effects, as needed, e.g., in the
description of quantum Brownian motion \cite{Di}, quantum optics
\cite{EL}, and semiconductor device simulations \cite{DNJ}. It can be
considered as a quantum mechanical generalization of the usual kinetic
Fokker-Planck equation (or Kramer's equation), to which it is known to
converge in the classical limit $\hbar \to 0$, after an appropriate
rescaling of the appearing physical parameters \cite{Bo}. The WFP
equation has been partly derived in \cite{CEFM} as a rigorous scaling
limit for a system of particles interacting with a heat bath of
phonons.  Additional ``derivations'' (based on formal arguments from
physics) can also can be found in \cite{De, Di, Va, VH}.

In recent years, mathematical studies of WFP type equations mainly
focused on the \emph{Cauchy problem} (with or without self-consistent
Poisson-coupling), see \cite{ACD, ADM1, ADM2, ALMS, ArSp, CLN}.  In
these works, the task of establishing a rigorous definition for the
particle density $n(t,x)$ has led to various functional analytical
settings.  To this end, it is important to note that the dynamics
induced by \eqref{qfp} maps $\mathscr T_1^+(L^2(\R^d))$ into itself,
since the so-called \emph{Lindblad condition} is fulfilled (see again
Remark \ref{rem1} below).  For more details on this we refer to
\cite{ALMS, ArSp} and the references given therein.  In the present
work we shall be mainly interested in the asymptotic behavior as $t
\to + \infty$ of solutions to \eqref{qfp}.  To this end, we first need
to study the stationary problem corresponding to \eqref{qfp}.  Let us
remark, that stationary equations for open quantum systems, based on the
Wigner formalism, seem to be rather difficult to treat as only very
few results exist (in spite of significant efforts, cf. \cite{ALZ}
where the stationary, inflow-problem for the linear Wigner equation in
$d=1$ was analyzed).  In fact the only result for the WFP equation is
given in \cite{SCDM}, where the existence of a unique steady state for
a quadratic potential $V(x)\propto |x|^2$ has been proved. However,
the cited work is based on several explicit calculations, which can
\emph{not} be applied in the case of a more general potential $V(x)$.

The goal of the present paper is \emph{twofold}: First, we aim to establish the
existence of a normalized steady state $w_\infty(x,\xi)$ for \eqref{qfp} in the case of confining potentials $V(x)$, which are given by a suitable class of perturbations of
quadratic potentials (thus, $V(x)$ can be considered as a perturbed harmonic oscillator potential).
The second goal is to study the long-time behavior of \eqref{qfp}. We shall prove exponential
convergence of the time-dependent solution $w(t,x,\xi)$ towards $w_\infty$ as $t\to +\infty$. In a subsequent step, we shall also prove that the
stationary Wigner function $w_\infty$ corresponds to a density
matrix operator $\rho_\infty \in \mathscr T_1^+$.  Remarkably, this proof
exploits the positivity preservation of the \emph{time-dependent}
problem (using results from \cite{ArSp}), via a stability
property of the steady states.

To establish the existence of a (unique) steady state $w_\infty$, the basic idea is to prove the existence of a spectral gap for the unperturbed Wigner-Fokker-Planck operator with
quadratic potential. This implies invertibility of the (unperturbed) WFP-operator on the orthogonal of its kernel.
Assuming that the perturbation potential is sufficiently small with respect to this spectral gap, we can set up a fixed point iteration to obtain the existence of $w_\infty$.
The key difficulty in doing so, is the choice of a suitable functional setting:
On the one hand a Gaussian weighted $L^2$--space seems to be a natural candidate, since it ensures dissipativity
of the unperturbed WFP-operator (see Section \ref{spre}). Indeed, this space is classical in the study of the long-time behavior of the classical (kinetic) Fokker-Planck equation, see \cite{HeNi}.
However, it  \emph{does not} allow for feasible
perturbations through $\Theta[V_0]$. In fact, even for smooth and compactly supported perturbation potentials $V_0$,
the operator $\Theta[V_0]$ would be \emph{unbounded} in such an $L^2$--space (due to the non-locality of $\Theta[V_0]$, see Remark \ref{unbounded}).
We therefore have to enlarge the functional space
and to show that the unperturbed WFP-operator then still has a (now smaller) spectral gap.
This is a key step in our approach. It is a result from spectral and semigroup theory
(cf. Proposition~\ref{gap_big_space}) which is related to a more general mathematical
theory of spectral gap estimates for kinetic equations, developed in parallel in \cite{GMM} (see also \cite{Mo}). 
We also remark that for $V(x)=|x|^2$ the WFP equation corresponds to a differential operator with  
quadratic symbol \cite{SCDM} and thus our approach is 
closely related to recent results for hypo-elliptic and sub-elliptic operators given in \cite{EH, HeNi,PS}.

Comparing our methods to closely related results in the quantum mechanical
literature, we first cite \cite{FR}, where several criteria for the existence of stationary density
matrices for quantum dynamical semigroups (in Lindblad form) were obtained by means of compactness methods. In \cite{AFN} the
applicability of this general approach to the WFP equation was established. In
\cite{FV, FR2} sufficient conditions (based on commutator relations for the
Lindblad operators) for the large-time convergence of open quantum
systems were derived.  However, these techniques do not provide a
rate of convergence towards the steady states. In comparison to that, the
novelty of the present work consists in establishing steady states in a kinetic framework
and in proving exponential convergence rates. However,
the optimality of such rates for the WFP equation remains an open problem. In this context one should also mention the recent work \cite{HeSj}, in
which explicit estimates on the norm of a semigroup in terms of bounds on the resolvent of its generator are obtained, very much along the same lines as in present paper and in \cite{GMM}.\\

The paper is organized as
follows: In Section~\ref{sset} we present the basic mathematical setting (in particular the
class of potentials covered in our approach) and state our two main theorems. In Section~\ref{spre} we collect some known results for the case of a purely quadratic potential and we introduce the
Gaussian weighted $L^2$--space for this unperturbed WFP
operator.
This basic setting is then generalized in Section~\ref{sec_new}, which contains the core of our (enlarged) functional framework:
We shall prove new spectral gap estimates for the WFP operator with
a harmonic potential in $L^2$--spaces with only \emph{polynomial
weights}. In Section~\ref{sbound} we prove the boundedness of the operator $\Theta[V_0]$ in these spaces.
Finally, Section~\ref{sproof} concludes the proof
of our main result by combining the previously established elements.
Appendix \ref{S7.1} includes the rather technical proof of a preliminary step which guarantees the applicability of the spectral method developed in \cite{GMM}.


\section{Setting of the problem and main results}\label{sset}

\subsection{Basic definitions} In this work we shall use the following convention for the Fourier
transform of a function $\varphi(x)$:
$$
\widehat \varphi (k) := \int_{\R^d} \varphi(x) \, \E^{- \I k \cdot x}
\D x.
$$
{}From now on we shall assume that the (real valued, time-independent) potential $V$, appearing in
\eqref{qfp}, is of the form
\begin{align}\label{Potential}
V(x)= \frac{1}{2}\, |x|^2 + \lambda{V}_0(x),
\end{align}
with $V_0\in C^\infty(\R^d;\R)$ and $\lambda \in\R$ some given \emph{perturbation parameter}. In
other words we consider a smooth perturbation $V_0$ of the harmonic
oscillator potential. The precise assumption on $V_0$ is listed in (\ref{condV}). An easy calculation shows that for such a $V$ the
stationary equation, corresponding to \eqref{qfp}, can be written as
\begin{align}\label{qfpS1}
L w  = \lambda\Theta[V_0]w,
\end{align}
where $L$ is the linear operator
\begin{align}\label{L}
 L w := -\xi \cdot\nabla _x w + x \cdot \nabla_\xi w + \Delta _\xi w
 +2 \, \textrm{div}_\xi (\xi w ) + \Delta _x w .
\end{align}
\begin{remark} When considering the
 slightly more general class of potentials
\begin{align*}
 V(x)= \frac{1}{2}\, |x|^2 + \alpha \cdot x + \lambda{V}_0(x),\quad
 \lambda\in \R,\,\alpha \in \R^d,
\end{align*}
we would find, instead of \eqref{L}, the following operator: $L_\alpha
w:= Lw + \alpha \cdot\nabla_\xi w$. Thus, by the change of variables
$x \mapsto  x + \alpha$ we are back to \eqref{L}.
\end{remark}
The basic idea for establishing the existence of (stationary) solutions
to (\ref{qfpS1}) is the use of a fixed point iteration. However, $L$
has a non-trivial kernel.  Indeed it has been proved in
\cite{SCDM} that, in the case $\lambda = 0$, there exists a unique
stationary solution $\mu\in\mathcal S(\mathbb R^{2d})$, satisfying
\begin{align}\label{steadyeq}
L \mu = 0
\end{align}
and the normalization condition
\begin{align}\label{mass}
\iint_{\R^{2d}} \mu (x,\xi) \, \D  x \, \D \xi = 1.
\end{align}
Explicitly, $\mu$ can be written as
\begin{equation}\label{mu}
\mu  = c\, \E^{- A(x,\xi)} ,
\end{equation}
where the function $A$ is given by
\begin{equation}
\label{A}
A(x,\xi) :=   \frac{1}{4}\, \left(|x|^2+2 x\cdot \xi +3 |\xi |^2
\right),
\end{equation}
and the constant $c>0$ is chosen such that
\eqref{mass} holds. Note that for any $\rho \in
\mathscr T_1$ such that $w\in L^1(\R^{2d})$ the following formal
identity
$$
\tr \rho = \int_{\R^d} \rho (x,x) \, \D x = \iint_{\R^{2d}} w(x,\xi)
\, \D x \, \D \xi,
$$
can be rigorously justified by a limiting procedure in $\mathscr T_1$,
see \cite{Ar}. Since $\tr \rho$ is proportional to the total mass of
the quantum system, we can interpret condition \eqref{mass} as a mass
normalization.

In the following, we shall denote by $\sigma >0$ the biggest constant
such that
\begin{equation}\label{conv_A} {\rm Hess} A-\sigma \,
 \emph{\textrm{\bf I}} \ge 0,\quad \textrm{for all}\;
 (x,\xi)\in\R^{2d},
\end{equation}
in the sense of positive definite matrices, where $\emph{\textrm{\bf
   I}}$ denotes the identity matrix on $\R^{2d}$.  In the analysis of
the classical FP equation, condition \eqref{conv_A} is referred to as
the \emph{Bakry-Emery criterion} \cite{AMTU}.  In our case one easily
computes $$\sigma = 1-1/\sqrt{2}.$$
In a Gaussian weighted $L^2$--space, $\sigma$ will be the spectral gap of the unperturbed WFP-operator and hence the decay rate towards the corresponding steady $\mu$ (cf. \eqref{spectrum-L},  \eqref{semigr1} below).\\

The functional setting of our problem will be based on the following
weighted Hilbert spaces. While the stationary and transient Wigner function are real valued,
we need to consider function spaces over $\C$, for the upcoming spectral analysis.

\begin{definition}\label{hmDef} For any $m \in \N$, we define ${\h}_m: =
 L^2(\R^{2d}, \nu^{-1}_m \D x \,\D \xi)$, where the weight is
$$\nu^{-1}_m := 1 + A^m(x,\xi).$$
We equip $\h_m$ with the inner
 product
\begin{align*}
{\langle f,g\rangle} _{\h_m } = \iint_{\R^{2d}}  \frac{f \bar g}{\nu_m} \, \D x \, \D \xi.
\end{align*}
Clearly, we have that $\h_{m+1} \subset \h_m$, for all $m\in \N$.
\end{definition}

\subsection{Main results}

With these definitions at hand, we can now state the main theorems of our work.
Note that for the sake of transparency we did not try to optimize the appearing constants.

\begin{theorem}\label{th1}
 Let $m\ge Kd$ be some fixed integer, where $K=K(A)\in (1,144]$ is a constant depending
 only on $A(x,\xi)$ (defined in Lemma \ref{techlemma}). Assume that the perturbation potential
 $V_0$ satisfies
\begin{align}\label{condV}
\Gamma_m :=    C_m \max_{ |j|\le  m } \| \partial_{x}^j V_0 \|_{L^\infty(\R^d)}<+\infty,
\end{align}
where $C_m >0$ depends only on $m$ and $d$, as seen in the proof of Proposition \ref{lempert}.
Next we fix some $\widetilde \gamma_m \in (0, \gamma_m)$, where $\gamma_m>0$ is given in \eqref{gamma-m}.
Furthermore, let the perturbation parameter $\lambda$ satisfy
\begin{align}\label{cond-lambda}
|\lambda| < \frac{\widetilde\gamma_m}{ \Gamma_m \delta_m}  \,,
\end{align}
where 
$\delta_m=\delta_m(\widetilde\gamma_m)> 1 $ is defined in \eqref{C}. Then it holds:
\begin{itemize}
\item[(i)] The stationary Wigner-Fokker-Planck equation \eqref{qfpS1} admits a unique weak solution $w_ \infty \in \h_m \, \cap \,  H^1(\R^{2d})$, satisfying $\iint_{\R^{2d}} w_\infty \, \D x \,\D \xi= 1$.
Moreover, $w_\infty$ is real valued and satisfies $w_\infty \in H^2_{\rm loc}(\R^{2d})$.
\item[(ii)] Equation \eqref{qfp} admits a unique mild solution $w\in C([0,\infty), \h_m)$. In addition, for any such mild solution $w(t)$ with initial data $w_0 \in \h_m$ satsifying $\iint_{\R^{2d}} w_0 \, \D x \,\D \xi= 1$, we have
$$
{\| \, w(t) - w_\infty \, \|}_{\h_m} \leq  \delta_m\E^{- \kappa_m t} \, {\| \, w_0 - w_\infty \, \|}_{\h_m}\, ,\qquad \forall \, t\ge0,
$$
with an exponential decay rate
$$
\kappa_m :=    \widetilde \gamma_m  -|\lambda| \delta_m\Gamma_m>0.
$$
\item[(iii)] Concerning the continuity of
$w_\infty=w_\infty(\lambda)$ w.r.t. $\lambda$, we have
$$
{\| \, w_\infty - \mu \, \|}_{\h_m} \leq \frac{|\lambda| \delta_m\Gamma_m}{\widetilde\gamma_m  -|\lambda| \delta_m\Gamma_m }\|\mu\|_{\h_m}.
$$
\end{itemize}
\end{theorem}
\begin{remark} In this result, the constant $\sigma_m:=\widetilde \gamma_m/\delta_m>0$, roughly speaking, plays the same role for $L$ on $\mathcal H_m$ as $\sigma>0$ does in the case of $\mathcal H$ (where $\mathcal H$ is defined in Definition \ref{definit:h}), where it is nothing but the size of the
spectral gap, see Proposition \ref{prop}. For $L$ on $\mathcal H$, $\sigma $ also gives the exponential decay rate in the unperturbed case $\lambda = 0$. For $L$ on $\mathcal H_m$ the situation is more complicated. Here,
assertion (ii) yields an exponential decay of the unperturbed semi-group with rate $\kappa_m = \widetilde \gamma_m \in (0, \gamma_m)$ and $\gamma_m \not = \sigma_m$ (but possibly equal to $\sigma$, as can be seen from \eqref{C}).
In addition, one should note that $\delta_m>1$ may blow-up as $\widetilde \gamma_m\nearrow\gamma_m$, cf. estimate \eqref{estim_line}.
\end{remark}
Theorem \ref{th1} is formulated in the Wigner picture of
quantum mechanics. We shall now turn our attention to the
corresponding density matrix operators $\rho(t)$. This is important
since it is {\it a priori} not clear that $w_\infty$ is physically
meaningful -- in the sense of being the Wigner transform of a positive trace class
operator. To this end we denote by $\rho_\infty$ the Hilbert-Schmidt
operator corresponding to the kernel $\rho_\infty(x,y)$, which is
obtained from $w_\infty (x,\xi)$ by the \emph{inverse Wigner
 transform}, i.e.
\begin{equation*}
\label{inw}
\rho_\infty (x,y)= \int_{\mathbb R^d} w_\infty\left(\frac{x+y}{2}, \xi
\right) \, \E^{ -\I\xi \cdot (x-y)} \, \D \xi.
\end{equation*}
Analogously we denote by $\rho_0$ the Hilbert-Schmidt operator corresponding to the initial Wigner function $w_0 \in \mathcal H_m$.

We remark that the existence of a unique mild solution of equation \eqref{qfp} on $\h_m$ will be a byproduct of our analysis.

\begin{theorem} \label{th2} Let $m\ge Kd$ be some fixed integer.
Let $V_0$, $\lambda$, and $w_0$ satisfy the same assumptions as in Theorem \ref{th1}.
Then we have:
\begin{itemize}
\item[(i)] The steady state $\rho_\infty$ is a positive trace-class operator on $L^2(\R^d)$, satisfying $\tr \rho_\infty = 1$.
\item[(ii)]
Let $\rho \in C([0,\infty), \mathscr T_2)$ be
the unique density matrix trajectory corresponding to the mild solution of \eqref{qfp}. Then, the steady state $\rho_\infty$ 
is exponentially stable, in the sense that
$$
{\| \, \rho(t) - \rho_\infty \, \|}_{\mathscr T_2} \leq  (2\pi)^{\frac{d}{2}}\delta_m \E^{- \kappa_m t} {\| \, w_0 - w_\infty \, \|}_{\h_m} \ , \qquad \forall \, t\ge0.
$$
\item[(iii)]
If the initial state $w_0 \in \mathcal H_m$ corresponds to a density matrix $\rho_0 \in \mathscr T_1^+$ (and hence $w_0$ is real valued,
$\tr \rho_0 \equiv \iint w_0 \, \D x \D \xi = 1$), then we also have
$$
\lim_{t \to \infty} { \| \, \rho(t)  - \rho_\infty \|}_{\mathscr T_1} = 0.
$$
\end{itemize}
\end{theorem}
Note that, in the presented framework, we do not obtain exponential
convergence towards the steady state in the ${\mathscr T_1}$-norm but
only in the sense of Hilbert-Schmidt operators. This is due to the weak compactness
methods involved in the proof of Gr\"umm's theorem (cf.\ the proof of Th.\ \ref{th2} in \S\ref{sproof}).
\begin{remark}\label{rem1}
Consider now the following, more general quantum
 Fokker-Planck type operator replacing the r.h.s.\ of \eqref{qfp}:
$$
Qw:= D_{\rm pp} \Delta_\xi w+2\, D_{\rm pq} \,  \text{div}_x \left(\nabla_\xi
 w\right) + 2 D_{\rm f} \, \text{div}_\xi\left(\xi w\right)
+ D_{\rm qq} \, \Delta_x w.
$$
It is straightforward to extend our results to this case
as long as the \emph{Lindblad condition} holds, i.e.
\begin{equation}\label{lindblad}
D_{\rm pp} \ge0,\qquad D_{\rm pp} D_{\rm qq}- \left(D_{\rm pq}^2+\frac{D_{\rm f}^2}{4}\right)\ge0.
\end{equation}
The modified quadratic function $A(x,\xi)$ is given in \cite{SCDM}. The Lindblad condition \eqref{lindblad} implies that discarding in \eqref{qfp} the diffusion in $x$, and
hence reducing the r.h.s.\ to the classical Fokker-Planck operator
$Q_{\rm cl}w:=\Delta_\xi w+2 \, \textrm{div}_\xi (\xi w)$, would \emph{not}
describe a ``correct'' open quantum system.  Nevertheless, this is a frequently used
model in applications \cite{ZP}, yielding reasonable results in numerical simulations.
\end{remark}


\section{Basic properties of the unperturbed operator $L$}\label{spre}

\subsection{Functional framework} It has been shown in \cite{SCDM} that the operator $L$, defined in \eqref{L}, can be rewritten
in the following form
\begin{align}\label{L1}
Lw= \text{div}  \left( \nabla w +w (\nabla A + F)\right),
\end{align}
with
\begin{equation}\label{identity}
 \diver(F\E^{-A}) = \frac1c\diver(F\mu)=0.
\end{equation}
Here and in the sequel, all differential operators act with respect to both
$x$ and $\xi$ (if not indicated otherwise).
In \eqref{L1}, the function $A$ is
defined by \eqref{A} and
\begin{align}
\label{fdef}
F :=
\begin{pmatrix}
-\xi \\
x+2  \xi
\end{pmatrix}
-\nabla A = \frac{1}{2}
\begin{pmatrix}
-x-3\xi \\
x+ \xi
\end{pmatrix}
.
\end{align}
The reason to do so is that \eqref{L1} belongs to a class
of non-symmetric Fokker-Planck operators considered in
\cite{AMTU}. From this point of view, a natural functional space to study the unperturbed operator $L$ is given by the following definition.

\begin{definition} \label{definit:h} Let $\h: = L^2(\R^{2d},\mu^{-1}\; \D x \,\D \xi)$, equipped with the
inner product
\begin{align*}
{\langle f,g\rangle} _{\h} = \iint_{\R^{2d}}  \frac{f \bar g}{\mu}\, \D x \, \D \xi.
\end{align*}
\end{definition}

We can now decompose $L$ into its symmetric
and anti-symmetric part in $\h$, i.e.
\begin{equation}\label{decomp}
L = L^{\rm s} + L^{\rm as},
\end{equation}
where
\begin{equation}\label{Ls}
L^{\rm s} w := \text{div} \left(\nabla w + w \nabla A \right),
\quad L^{\rm as} w := \text{div} (F w).
\end{equation}
It has been shown in \cite{SCDM}, that the following
property holds:
\begin{align}\label{kernel}
L^{\rm s} \mu =0,  \quad L^{\rm as} \mu =0.
\end{align}
where $\mu$ is the stationary state defined in (\ref{mu}). Next we shall properly define the operator $L$.
To this end we first consider ${L \big |_{C_0^\infty}}$, which is closable (w.r.t.\ the $\h$-norm)
since it is dissipative:

\begin{lemma} \label{disslemma} $L\big |_{C_0^\infty}$ is dissipative,
 i.e.\ it satisfies $\emph{Re} \,{\langle L w , w \rangle} _{\h} \leq 0$, for all
 $w \in C_0^\infty(\R^{2d})$.
\end{lemma}
\begin{proof}
Using $\nabla A = - \mu ^{-1} \nabla \mu$  we have, on the one hand
\begin{align*} {\langle L^{\rm s} w , w \rangle} _{\h} = & \,
 \iint_{\R^{2d}} \frac{\bar w}{\mu} \, \text{div}\left(\nabla w + w \nabla
   A \right) \D x\, \D \xi
 = \, \iint_{\R^{2d}} \frac{\bar w}{\mu} \, \text{div}\left(\mu \nabla \left(\frac{w}{\mu} \right) \right) \D x\, \D \xi \\
 = & \, - \iint_{\R^{2d}} \mu \left |\nabla \left(\frac{w}{\mu}
   \right) \right|^2 \D x\, \D \xi \leq 0.
\end{align*}
On the other hand, it follows from \eqref{identity} that
$$
w \, \text{div} F  = - \frac{w}{\mu}  \, F \cdot \nabla \mu,
$$
and thus
$$
\text{div} (F w ) = - \mu F  \cdot \left( \frac{w}{\mu^2} \nabla \mu - \frac{\nabla w}{\mu} \right) =
\mu F \cdot \nabla \left(\frac{ w}{\mu}\right).
$$
An easy calculation then shows
\begin{align*}
 \text{Re}\,{\langle L^{\rm as} w , w \rangle} _{\h} = & \,\text{Re}\,
 \iint_{\R^{2d}} \frac{\bar w}{\mu} \, \text{div} (F w) \, \D x\, \D \xi =
 \text{Re}\,\iint_{\R^{2d}} \frac{\bar w}{\mu} \, F \cdot \nabla \left( \frac{w}{\mu}
 \right) \mu \,  \D x\, \D \xi  \\
 = & \, - \frac{1}{2} \iint_{\R^{2d}} \left|\frac{w}{\mu} \right|^2
 \text{div} (F \mu) \, \D x\, \D \xi = 0,
\end{align*}
by \eqref{identity}.  To sum up we have shown that $\text{Re}\,{\langle L w , w
 \rangle} _{\h} \leq 0$ holds.
\end{proof}
The operator $L\equiv \overline {L \big |_{C_0^\infty}}$ is now
closed, densely defined on $\h$ and dissipative.  Moreover one easily
sees that $L^* = L^{\rm s} - L^{\rm as}$.  The main goal of this
section is to prove that $L$ admits a
spectral gap and is invertible on the orthogonal complement of its kernel. For the first property, we start showing that $L$ is the
generator of a $C_0$-semigroup of contractions on $\h$. For this we
recall the following result from \cite{ArSp}.

\begin{lemma}\label{quadlemma} Let the operator $P = p_2\, ( x ,\xi, \nabla_x, \nabla_\xi)$, where $p_2$ is a second order polynomial,
be defined on the domain $\mathscr D(P)=C_0^\infty(\R^{2d})$.
Then $P$ is closable and $\overline{P \big |_{C_0^\infty}}$ is the maximum extension of $P$ in $L^2(\R^{2d})$.
\end{lemma}

Several variants of such a result (on different functional spaces) can be found in \cite{ACD, ADM1}. We can use this result now in order to prove that
$L$ is the generator of a $C_0$-semigroup.

\begin{lemma} \label{lemsemi}
$L$ generates a $C_0$-semigroup of contractions on $\h$.
\end{lemma}
\begin{proof}
Defining $v:= w/\sqrt{\mu}$ transforms the evolution problem
$$
\partial_t w=L w,\quad w\big |_{t=0}=w_0\in\h
$$
into its analog on $L^2(\R^{2d})$. The new unknown
$v(t,x,\xi)$ then satisfies the following equation
\begin{equation*}
\partial_t v = H v, \quad v \big |_{t=0} =   {w _0}/ {\sqrt{\mu}} ,
\end{equation*}
where $H$ is (formally) given by
\begin{align*}
H v= \Delta v + F \cdot \nabla v + U v,
\end{align*}
and the new ``potential'' $U=U(x,\xi)$ reads
$$
U=\frac{1}{2} \Delta A - \frac{1}{4} |\nabla A|^2.
$$
Defining $H$ on $\mathscr D(H)=C_0^\infty(\R^{2d})$, we have
$$L w = \sqrt{\mu} \, H \left(\frac{w}{\sqrt{\mu}}\right),
$$
and thus the dissipativity of $L$ on $\h$ directly carries over to
$H\equiv \overline {H \big |_{C_0^\infty}}$ on $L^2(\R^{2d})$.  Next,
we consider $H^* \big |_{C_0^\infty}$, defined via $ {\langle H f , g
 \rangle} _{L^2} = {\langle f ,H^* g \rangle} _{L^2}$, for $f,g \in
C_0^\infty(\R^{2d})$.  Due to the definitions \eqref{A} and
\eqref{fdef}, the operator $H^*\big |_{C_0^\infty}$ is exactly of the
form needed in order to apply Lemma \ref{quadlemma}.
Thus, $H^* \equiv \overline{H^*\big |_{C_0^\infty}}$ is also
dissipative (on all of its domain). Hence,
the Lumer-Phillips Theorem (see \cite{Pa}, Section 1.4) implies that $H$ is the generator of a
$C_0$-semigroup on $L^2(\R^{2d})$, denoted by $\E^{H
 t}$.

Reversing the transformation $w\to v$ then implies that $L$ is the generator of the $C_0$-semigroup $U_t$ on $\h$, given by
$$
U_t w_0 = \sqrt{\mu} \, \E^{H t} \left(\frac{w_0}{\sqrt{\mu}}\right).
$$
This finishes the proof.
\end{proof}

\subsection{Semigroup properties on $\h$}
The above lemma shows that the unperturbed WFP equation
$$
\partial_t w=L w,\quad w\big |_{t=0}=w_0 \in \h,
$$
has, for all $w_0\in\h$, a unique mild solution $w\in C([0,\infty),\h)$, where $w(t,x,\xi)= U_t w_0(x,\xi)$, with $U_t$ defined above.
Obviously we also have $U_t \mu = \mu$, by \eqref{steadyeq}.
Moreover, in \cite{SCDM} the Green's function of $U_t$ was computed explicitly. It shows that
$U_t$ conserves mass, i.e.
$$
\iint_{\R^{2d}} w(t, x,\xi) \, \D x \, \D \xi = \iint_{\R^{2d}}
w_0(x,\xi) \, \D x \, \D \xi, \quad \forall \, t \geq 0.
$$
Next, we define
\begin{equation}\label{perpspace}
\h^\perp:= \{ w \in \h: w \perp \mu \}\subset \h,
\end{equation}
which is a closed subset of $\h$. Note that $w \perp \mu$ simply means
that
$$
{\langle w, \mu \rangle}_\h \equiv \iint_{\R^{2d}} w (x,\xi) \, \D x \, \D \xi = 0.
$$
Hence, we have for $w\in C_0^\infty(\R^{2d})$, using \eqref{Ls}:
$$
 {\langle L^{\rm as}w, \mu \rangle}_\h \equiv \iint_{\R^{2d}}L^{\rm as} w (x,\xi) \, \D x \, \D \xi = 0.
$$
Thus, $L^{\rm as}:\,\h^\perp\cap \mathscr D(L^{\rm as})\to \h^\perp$. Moreover,
$L^{\rm s}:\,\h^\perp\cap \mathscr D(L^{\rm s})\to \h^\perp$, since $\h^\perp$ is spanned by the eigenfunctions of $L^{\rm s}$ (except of $\mu$).
To sum up, the operators $L^{\rm s}$ and $L^{\rm as}$ are simultaneously reducible on the two subspaces
$\h=\text{span}[\mu]\oplus\h^\perp$.

We also have that $U_t$ maps $\h^\perp$ into itself, since
for $w_0 \in \h^\perp$ the conservation of mass implies
\begin{equation}\label{mass-conserv}
{\langle U_t w_0, \mu \rangle}_\h \equiv \iint_{\R^{2d}} w(t, x,\xi) \, \D x \, \D \xi =
\iint_{\R^{2d}} w_0(x,\xi) \, \D x \, \D \xi = 0, \quad \forall \, t\geq 0.
\end{equation}
Lemma \ref{lemsemi} allows us to prove that $L$ has a spectral gap in
$\h$, in the sense that
\begin{equation}\label{spectrum-L}
 \sigma(L)\setminus\{0\} \subset \{z\in\mathbb{C}: \text{Re}\,z\le -\sigma\}.
\end{equation}
\begin{proposition} \label{prop}
It holds
$$
{\| L ^{-1} \|}_{\mathscr B(\h ^\perp)} \leq \frac{1}{\sigma}\, ,
$$
where $\sigma >0$ is defined in \eqref{conv_A}.
\end{proposition}

\begin{proof}
 Condition \eqref{conv_A} implies that $L^{\rm s}$
 has a spectral gap of size $\sigma >0$ (cf.\ \S3.2 in \cite{AMTU}, e.g.). Moreover, \cite[Theorem 2.19]{AMTU} also yields exponential decay (with the same rate) for the non-symmetric WFP equation:
\begin{align} \label{semigr1}
{\|\, U_t(w_0-\mu) \|}_{\h} \le {\E^{-\sigma t}\| \, w_0-\mu \, \|}_{\h}.
\end{align}
Here, $w_0\in\h$ has to satisfy $ \iint _{\R^{2d}} w_0 \; \D x \D \xi = \iint
_{\R^{2d}} \mu \; \D x \D \xi =1 $. By the discussion above, we know that $L\big |_{\h^\perp}$ is the generator of $U_t\big |_{\h^\perp}$. Hence, \eqref{semigr1} implies
\begin{align} \label{resolvent}
{\| (L-z)^{-1}\, \|}_{\mathscr{B}(\h^{\perp})} \le
\frac{1}{\text{\rm Re} \, z + \sigma},\quad \forall\; z \in \mathbb{C}, \ \text{Re$ \, z > -\sigma$,}
\end{align}
which proves the assertion for $z=0$.
\end{proof}


As a final preparatory step in this section, we shall prove more detailed coercivity properties of $L$ within $\h^\perp$. We shall denote $\h^1 := \{ w \in \h\, :\, \nabla w \in \h\}$, and $\h^{-1}$ will denote its dual.

\begin{lemma}\label{coercL}
In $\h^\perp$ the operator $L$ satisfies
\begin{equation}\label{coerc-0}
- \emph{Re}\,\left\langle Lw, w \right\rangle_{\mathcal H}
\ge \sigma \, \| w \|_{\mathcal H}^2.
\end{equation}
Similarly, there exists a constant $0< \alpha < \sigma$, such that
\begin{equation}\label{coerc-1}
- \emph{Re}\,\left\langle Lw, w \right\rangle_{\mathcal H}
\ge \alpha \, \| w \|_{\mathcal H^1}^2,\quad\forall\,w\in\h^\perp\cap\h^1.
\end{equation}
\end{lemma}

\begin{proof}
 We shall use the weighted Poincar\'e inequality (see \cite{AMTU}): For any
 function $f \in L^2(\R^{2d},\mu \D x \, \D\xi)$, such that $\iint_{\R^{2d}} f \mu\, \D x \, \D
 \xi=0$, it holds:
\begin{align}\label{Poinc}
 \iint_{\R^{2d}} |f|^2 \, \mu \; \D x \, \D \xi \le \frac{1}
 {\sigma} \, \iint_{\R^{2d}} \mu \, |\nabla f|^2 \; \D x \, \D \xi.
\end{align}
Estimate (\ref{coerc-0}) then readily follows by setting $f= w / \mu $:
$$
\text{Re}\,\langle L w, w \rangle_{\mathcal H}   = - \iint_{\R^{2d}} \mu \left |
 \nabla \left( \frac{w}{\mu}\right)\right|^2\; \D x \, \D \xi \le
-\sigma \, \iint_{\R^{2d}} \frac {|w|^2}{\mu}\; \D x \, \D \xi.
$$

In order to prove assertion (\ref{coerc-1}), we note that
$$
\mu \left | \nabla \left( \frac{w}{\mu} \right)\right | ^2 = \frac{
  |\nabla w|^2}{\mu} -2d \, \frac { |w|^2}{\mu} - \text{div} \left( |w|^2
\,    \frac{\nabla \mu}{\mu^2}\right),
$$
taking into account that $\Delta (\log \mu) = -2d $. Next, let $0 <\alpha < 1 $ (to be chosen later), and write
\begin{align*}
\text{Re}\,\langle L w, w \rangle_{\mathcal H}
=& - \alpha\iint_{\R^{2d}} \mu \left | \nabla \left(
    \frac{w}{\mu}\right)\right|^2\; \D x \, \D \xi - (1-\alpha) \,
\iint_{\R^{2d}} \mu \left | \nabla \left(
    \frac{w}{\mu}\right)\right|^2\; \D x \, \D \xi \\
= & - \alpha \, \iint_{\R^{2d}} \frac{ |\nabla w |^2}{\mu} \; \D x \, \D
\xi
+2d \, \alpha \, \iint_{\R^{2d}} \frac { |w|^2}{\mu}\; \D x \, \D \xi \\
& - (1-\alpha) \, \iint_{\R^{2d}} \mu \left | \nabla \left(
   \frac{w}{\mu}\right)\right|^2\; \D x \, \D \xi.
\end{align*}
Inequality (\ref{Poinc}) for $f=w/\mu$ then implies:
$$
- (1-\alpha)\iint_{\R^{2d}} \mu \left | \nabla \left(
   \frac{w}{\mu}\right)\right|^2\; \D x \, \D \xi \le
- \sigma (1-\alpha) \, \iint_{\R^{2d}}\frac{|w|^2}{\mu} \; \D x \, \D \xi.
$$
Therefore
\begin{align*}
\text{Re}\,\langle L w, w \rangle_{\mathcal H} \le  - \alpha \, \iint_{\R^{2d}}
\frac{ |\nabla w |^2}{\mu} \; \D x \, \D \xi  +
\left( 2d \, \alpha -  \sigma \, (1-\alpha) \right) \,
\iint_{\R^{2d}} \frac { |w|^2}{\mu}\; \D x \, \D \xi .
\end{align*}
The choice $ \alpha=\sigma/ (\sigma+2d+1)$ yields assertion (\ref{coerc-1}).
\end{proof}

In the next section we shall study the operator $L$ in the larger functional spaces $\h_m$ (see Definition  \ref{hmDef}).
This is necessary since the perturbation operator $\Theta[V_0]$ is unbounded in $\h$, even for $V_0\in C_0^\infty(\R^d)$, cf. Remark \ref{unbounded}.


\section{Study of the unperturbed problem in $\h_m$}\label{sec_new}

In this section, we adapt the general procedure outlined in \cite{GMM, Mo} to the specific model at hand. One of the main differences to the models studied in \cite{GMM} is
the fact that the WFP operator includes a diffusion in $x$. Nevertheless, we shall follow the main ideas of \cite{GMM}. In a first step, this requires
us to gain sufficient control on the action of $U_t$ on $\h_m$. After that, we establish a new decomposition of $L$ (not to be confused with the decomposition $L = L^{\rm s} + L^{\rm as}$ used above)
in order to lift resolvent estimates onto the enlarged space $\h_m\supset \h$. Together with the Gearhart-Pr\"uss Theorem (cf.\ Theorem V.1.11 in \cite{EN}), these estimates
will finally allow us to infer exponential decay of $U_t$ on $\h_m$.

\subsection{Mathematical preliminaries} In \eqref{Ls} we decomposed the unperturbed evolution operator as $L = L^{s} + L^{\rm as}$.
As a first, from basic property of the spaces $\h_m$ (see Definition \ref{hmDef}), we note that
$L^{\rm as}$ is still anti-symmetric in $\h_m$, $m\in \N$.
\begin{lemma}
It holds
\begin{align}\label{antiH}
\emph{Re}\,\langle L^{\rm as} w, w \rangle_{ {\mathcal H}_m} =0, \quad \forall \, m \in \N.
\end{align}
\end{lemma}
\begin{proof}
A straightforward calculation yields $\diver(F) =0$. Hence
\begin{align*}
 \text{Re}\,\left\langle L^{\rm as} w , w \right\rangle_{\h_m} =
 & \,\text{Re}\, \iint_{\R^{2d}} \frac{\bar w}{\nu_m} \, \text{div} (F \, w)
 \, \D x\, \D \xi  = \,
 \frac{1}{2} \, \iint_{\R^{2d}} F \cdot \nabla \left(|w|^2\right) \,
 \nu_m^{-1}  \, \D x\, \D \xi  \\
 = & \ -\frac{1}{2} \, \iint_{\R^{2d}} |w|^2
 \, m \, A^{m-1} \, F \cdot \nabla A \, \D x\, \D \xi ,
\end{align*}
after integrating by parts and using $\nu_m^{-1} =
1+A^m(x,\xi)$. Now, it is easily seen from \eqref{fdef} that $F\cdot
\nabla A=0$, which implies \eqref{antiH}.
\end{proof}
The proof shows that, in the definition of $\h_m$, it is important to
choose the weight $\nu_m$ as a (smooth) function of
$A = - \log \frac\mu{c} $. Otherwise the fundamental property \eqref{antiH} would no
longer be true. Also note that in contrast to $L^{\rm as}$, the operator $L^{\rm s}$ is \emph{not} symmetric in $\h_m$. 
Before studying further properties of $L$ in $\h_m$ we state the following technical lemma. 
In order to keep the presentation simple, we shall not attempt to give the optimal constants.

\begin{lemma} \label{techlemma}

Let $A=-\log \frac\mu{c}$, as given in \eqref{A}. Then the following properties hold:
\begin{itemize}
 \item[(a)]
There exists a constant $a_1>0$, such that for all $m\in \N$ it holds:
  \begin{align*} \label{ass1}
    a_1 \, (1+A^m)\le A^{m-1} \, \left|\nabla A\right|^2, \quad \textrm{for all} \;|x|^2+|\xi|^2 \ge 12.
  \end{align*}
\item[(b)] \label{ass3} 
  Choosing $K:=\frac{4}{a_1}$ it holds for all integer $m\ge Kd$:
  \begin{align*}  4d \, \left( 1+A^m \right) \le m  A^{m-1} \,
    \left|\nabla A\right|^2, \quad \textrm{for all} \; |x|^2+|\xi|^2 \ge 12.
  \end{align*}
\item[(c)] \label{ass2} There exists a constant $a_2>1$ such that
  \begin{align*}
    |\nabla A|^2 \le a_2 \, A, \quad \forall\,x,\, \xi \in \R^{d}.
  \end{align*}
\item[(d)] \label{ass4} For any $|x|,|\xi| \ge \frac{1}{\varepsilon}$ and $m\ge 1$, it
  holds
  \begin{align*}
    \Delta \left(1+A^m\right) \le m \, A^{m-1} \, \left|\nabla
      A\right|^2 \, \varepsilon^2 \, 6(m-1+3d).
  \end{align*}
\end{itemize}
\end{lemma}

\begin{proof}
Using Young's inequality we easily obtain
\begin{equation}\label{A-bound}
\frac{1}{12}(|x|^2+|\xi|^2) \le A(x,\xi) \le |x|^2+|\xi|^2,
\end{equation}
\begin{equation}\label{gradA-bound}
\frac{1}{18}(|x|^2+|\xi|^2) \le |\nabla A(x,\xi)|^2 \le 3(|x|^2+|\xi|^2).
\end{equation}
This yields assertion (c). To show (a), we note from \eqref{A-bound} that
$$
 1\le A\le A^m, \quad \forall\,|x|^2+|\xi|^2 \ge 12.
$$
Hence, we obtain with \eqref{A-bound}, \eqref{gradA-bound}:
\begin{align*}
 1+A^m  \le 2 \, A^m\le 36\,A^{m-1}|\nabla A|^2,
\end{align*}
which is assertion (a). We further note that assertion (b) is a direct consequence of (a). Finally, to prove assertion (d), we compute
\begin{align*}
 \Delta \left(1+A^m\right) & = m \, A^{m-1} \, \left| \nabla A
 \right|^2 \, \left( \frac{m-1}{A} + \frac{2d}{ |\nabla A|^2} \right)\\
 & \le m \, A^{m-1} \, \left| \nabla A \right|^2 \, \varepsilon^2 \,
 {6 \, (m-1+3d)},\quad \textrm{for}\; |x|, \, |\xi|\ge
 \frac{1}{\varepsilon}.
\end{align*}\end{proof}

\begin{remark}\label{Rem4.4}
Note that the constants $a_1\ge \frac1{36}, a_2\le36$, and $K=\frac4{a_1}\le144$ can be chosen \emph{independent} of $m\in\N$ and of the spatial dimension $d\in \N$. Moreover, $K=1$ is not possible for $d=1$.
\end{remark}

\subsection{Semigroup properties on $\h_m$}
Analogously to \eqref{perpspace}, we now define the following closed subset of $\h_m$:
$$
\h_m^\perp:= \left\{ w \in \h_m: w \perp \nu_m \right\}, \quad m \in \N,
$$
which is again characterized by the zero-mass condition
\begin{equation}\label{null}
{\langle w, \nu_m \rangle}_{\h_m} \equiv \iint_{\R^{2d}} w (x,\xi) \, \D x \, \D \xi = 0.
\end{equation}
Thus we have $\h^\perp\subset \h_m^\perp\; \forall\,m\in\N$.
As before, we define $L$ on $\h_m$ via $L\equiv \overline {L \big |_{C_0^\infty}}$ which yields a closed, densely defined operator on $\h_m$ for each $m\in \N$.
In addition, we also have the following result.

\begin{lemma} \label{lemsemi-m} For each $m\in \N$, the operator $L$
 generates a $C_0$-semigroup of bounded operators on $\h_m$,
 satisfying
\begin{equation}\label{boundsemi}
\left\| U_t  \right\|_{\mathscr B(\h_m)}
\le \E^{ \beta_m \, t},\quad \beta_m\in \R.
\end{equation}
\end{lemma}
\begin{proof} We compute
\begin{align*}
 \text{Re}\,\langle Lw,w \rangle_{{\h_m}} =
 & \ - \iint_{\R^{2d}} \frac{ |\nabla w|^2}{\nu_m} \; \D x \, \D \xi \\
 & \ + \frac{1}{2} \iint_{\R^{2d}} |w|^2 \, \left(\Delta A^m -
    m \,
   A^{m-1} \,
 \left|\nabla A\right|^2 +\frac{ 2d}{\nu_m} \right) \; \D x \, \D \xi,
\end{align*}
by taking into account (\ref{antiH}) and the fact that $\nu_m^{-1} = 1
+ A^m(x,\xi)$. Using assertion (c) of Lemma \ref{techlemma}, we can
estimate
$$
\Delta (A^m) = m \, A^{m-2} \, \big( (m-1) \, \left|\nabla A\right|^2
 + 2d \, A \big) \le m \, \big( (m-1) \, a_2 +2d \big) \, A^{m-1}.
$$
Moreover, since, for all $x,\xi \in \R^d$: $A^{m-1}(x,\xi)\le
1+A^m(x,\xi)$, we consequently obtain
\begin{align*}
 \Delta (A^m) - m \, A^{m-1} \, \left|\nabla
   A\right|^2 +\frac{2d}{ \nu_m} \le \Delta (A^m) +\frac{ 2d}{
   \nu_m} \le \beta_m (1 + A^m),
\end{align*}
where
\begin{align*}
\beta_m := 2d + m \Big( (m-1)a_2  + 2d \Big).
\end{align*}
In summary, this yields
\begin{align*}
\text{Re}\,\left\langle Lw,w \right\rangle_{{\h_m}} \le  \beta_m \, \| w \|^2_{\h_m}.
\end{align*}
Thus, for the unperturbed evolution equation
$\partial _t w = L w$ we infer
$$
\frac{\D}{\D t} \| w \|^2_{\h_m} = 2 \, \text{Re}\,\left\langle Lw,w
\right\rangle_{{\h_m}} \le 2 \, \beta_m \, \| w \|^2_{\h_m},
$$
and the assertion follows.
\end{proof}

\begin{remark}\label{Rem4.5}
Note that $\beta_m$ \emph{cannot be negative} in Lemma \ref{lemsemi-m} since $U_t(\mu)=\mu$.
However, using some refined estimates below, we
shall find (see Proposition \ref{gap_big_space})
that the restricted semigroup $U_t\big|_{\h_m^\perp}$ is exponentially decaying, provided $m\in \N$ is sufficiently large.
To this end, we note that the two (non-orthogonal) subspaces $\h_m=\text{span}[\mu]\oplus\h_m^\perp$ are invariant under $L$ and under $U_t$ ($\forall\,m\in \N$) due
to mass conservation \eqref{mass-conserv} and \eqref{null}.
\end{remark}

As a final preparatory step, we shall need the following decomposition result for $L$, where we denote $$
 {\h}_m^1 := \{ w \in \h_m : \nabla w\in \h_m \}.
$$

\begin{proposition} \label{prop_assum_AB}
Let $m\ge Kd$ be some fixed integer, and $K$ was defined in Lemma \ref{techlemma}. Then there exists an $0<\e<1$ such that the
operator $L$ can be split into $L = L_1^\varepsilon+L_2^\varepsilon$, with $L^\e_1, L_2^\e$ defined in (\ref{L1varepsilon}) and (\ref{L2varepsilon}) and satisfying:
\begin{enumerate}
\item\label{prop_assum_AB-1} $L_1^\varepsilon : \h_m \to \h_m$ is a closed and
 unbounded operator, while $L_1^\varepsilon: \h_m^1 \to \h$ and
 $L_1^\varepsilon: \h_m \to \h^{-1}$ are bounded operators.

\item $(L_2^\varepsilon-z) : \mathcal H \to \mathcal H$ and
 $(L_2^\varepsilon-z) : \h_m \to \h_m$ are closed,
 unbounded and invertible operators for every $z \in \Omega:=\{ z\in \C : {\rm Re}\, z > - \Lambda_m\}$, where $\Lambda_m>0$ is a positive
 constant defined in \eqref{Lambda_m}.
\item The operator
 $$L_1^\varepsilon(L_2^\varepsilon-z)^{-1}:  \h_m \to \h\subset \h_m $$
   is bounded for any $z \in \Omega$.
\end{enumerate}
\end{proposition}

The proof is lengthy and rather technical and therefore deferred to Appendix A.

\begin{remark} Note that in Proposition \ref{prop_assum_AB}, $\varepsilon$ has to be chosen positive, in order to ensure assertion
(\ref{prop_assum_AB-1}).
In fact, while $L_2^\varepsilon$ continues to be coercive also for $\varepsilon =0$, the operators $L_1^\varepsilon: \h_m^1\to \h$ and $L_1^\varepsilon:\h_m\to\h^{-1} \;$ become unbounded as $\varepsilon \to 0$.
The fact that $L_1^\varepsilon$ is bounded for $\e>0$ is essential, in order to obtain the decay estimate \eqref{semigr2}, cf. the proof of Proposition \ref{gap_big_space}.
\end{remark}

Indeed, introducing the decomposition $L = L_1^\varepsilon+L_2^\varepsilon$ is one of the key ideas in \cite{GMM, Mo} in order to lift estimates for the resolvent  $R(z)=(L-z)^{-1}$ onto the larger space $\mathcal H_m$.
The general decomposition procedure introduced in \cite{GMM} applies to the  WFP equation and provides the following exponential decay of $U_t$ on $\h_m$.

\begin{proposition}\label{gap_big_space} Let $\sigma >0$ be the spectral gap of $L^{\rm s}$ in $\h$, and let $w_0\in \h_m$ with $\iint_{\R^{2d}}w_0\,\D x\D \xi=1$.
Then, for every integer $m\ge Kd$, it holds
\begin{align}\label{semigr2}
\| U_t(w_0-\mu)\|_{ {\h_m }} \le \delta_m  \E^{-\widetilde\gamma_m t} \| w_0-\mu\|_{\h_m},
\end{align}
for any $\widetilde\gamma_m \in(0,\gamma_m)$,
where
\begin{align}\label{gamma-m}
\gamma_m := \min\{\Lambda_m\,;\,\sigma\}> 0,
\end{align}
and $\delta_m=\delta_m(\widetilde\gamma_m)>1$ is given in \eqref{C}. Furthermore, we have for the resolvent set
\begin{align}\label{resolvent_m}
\varrho\big(L  \big|_{\h_m}\big) \supseteq \Omega_1:=\{ z \in \C: \text{\rm Re} \, z > - \gamma_m ,\,z\ne0 \}.
\end{align}
\end{proposition}

\begin{proof}
For the sake of completeness we briefly present the proof which follows the ones of Theorem~2.1, Theorem~3.1, and Theorem~4.1 in \cite{GMM}. The spirit of the proof is the following:
By using the operator factorization from Proposition \ref{prop_assum_AB}, we shall infer an estimate for the resolvent on $\h_m$.
Restricting the resolvent to $\h_m^\perp$ removes its singularity at $z=0$ and consequently yields a uniform estimate on the complex half plane $\{z \in \C :\ \text{Re}\,z\ge -\widetilde \gamma_m\}.$
The Gearhart-Pr\"uss Theorem \cite{G78,J84} then yields the exponential decay of $U_t$ on $\h_m^\perp$.
The proof now follows in several steps:

\emph{Step 1:} Following \cite{Mo06}  we define on ${\h}_m$ the operator
\begin{align}\label{inverse_operator}
R(z) := ( L_2^\varepsilon - z)^{-1} - \left( (L-z)\big|_{\h}\right)^{-1} L_1^\varepsilon  ( L_2^\varepsilon - z)^{-1},
\quad z \in \Omega_1.
\end{align}
Theorem 2.1 and Remark 2.2 in \cite{GMM} implies that $R(z)$ is the inverse operator of $(L-z)$ in ${\h}_m$ for any $z \in \Omega_1$ and therefore
statement (\ref{resolvent_m}) holds.
\\

\emph{Step 2:} The next step is devoted to obtaining uniform estimates for $R(z)=\left( (L-z)\big|_{\h_m}\right)^{-1}$, for $z\in\C$ on some appropriately defined half planes, cf. \cite[Theorem 3.1 (4)]{GMM}.
To this end, we shall first prove the following bound for the resolvent on $\h$:
\begin{align}\label{stup_estim}
\sup_{s}\| (L-(a+is))^{-1} \| _{\mathscr{B}(\h)} = K_0 < \infty,\quad \forall\,a\in (-\sigma,0).
\end{align}
Indeed, the constant $K_0=K_0(\sigma,a)$ can be explicitly obtained by considering the resolvent equation for  $\text{Re}\,z>-\sigma$ and $z\ne0$:
$$
 (L-z)f=g\quad \mbox{on } \h,
$$
Using the orthogonal decomposition $f=f^\perp+c_1\mu$, $g=g^\perp+c_2\mu$, having in mind that $L$ maps $\h^\perp$ into $\h^\perp$, we infer
$$
 f^\perp=(L-z)\big|_{\h^\perp}^{-1}\,g^\perp\,,\quad c_1=-\frac{c_2}{z}.
$$
In view of \eqref{resolvent} this yields
$$
 \|f\|_\h^2 \le \frac{1}{(\text{Re}\,z+\sigma)^2} \|g^\perp\|_\h^2 + \frac{|c_2|^2}{|z|^2}
 \le \max \left\{\frac{1}{(\text{Re}\,z+\sigma)^2}\,;\,\frac{1}{|z|^2}\right\} \|g\|_\h^2.
$$
Hence,
$$
 \|(L-z)^{-1} \|_{{\mathscr{B}(\h)}}
 \le \max \left\{\frac{1}{\text{Re}\,z+\sigma}\,;\,\frac{1}{|z|}\right\}
 \quad\mbox{for }\text{Re}\,z>-\sigma,\,z\ne0;
$$
and $K_0\le\max\left\{\frac{1}{a+\sigma}\,;\,\frac{1}{|a|}\right\}$.

From this bound on $(L-z)\big|_\h^{-1}$ we can deduce a bound on $(L-z)\big|_{\h_m}^{-1}$ for
$z\in\Omega_1\equiv \{ z \in \C: \text{\rm Re} \, z > - \gamma_m ,\,z\ne0 \}$. Using
(\ref{inverse_operator}), (\ref{coercive3}), and $L_1^\varepsilon\in\mathscr{B}(\h_m^1 \to \h)$, we infer
\begin{equation} \label{estim_line}
\begin{split}
&\!\!\!\!\!\!\!\!\!\!\!\!\!\!\!\!\|(L-z)^{-1} \|_{{\mathscr{B}(\h_m)}} \\
&\le \| (L_2^\varepsilon -z)^{-1} \|_{{\mathscr{B}(\h_m \to \h_m^1)}}
\left( 1+  \|(L-z)^{-1} \|_{{\mathscr{B}(\h)}} \,
\| L_1^\varepsilon \|_{{\mathscr{B}(\h_m^1 \to \h)}}\right) \\
&\le \max \left\{\frac{1}{\text{Re}\,z+\Lambda_m}\,;\,\frac{1}{\Lambda_m}\right\}
\left( 1+ \max \left\{\frac{1}{\text{Re}\,z+\sigma}\,;\,\frac{1}{|z|}\right\}
\| L_1^\varepsilon \|_{{\mathscr{B}(\h_m^1 \to \h)}}\right) \\
&=: \vartheta_m(z).
\end{split}
\end{equation}
Next we consider the resolvent of $L\big|_{\h_m^\perp}$.
First we note that both subspaces of $\h_m=\text{span}[\mu]\,\oplus\, \h_m^\perp$ are invariant for $(L-z)^{-1}$  (cf.\ Remark \ref{Rem4.5}). Hence $(L-z)\big|_{\h_m}^{-1}$ and $(L-z)\big|_{\h_m^\perp}^{-1}$ coincide on $\h_m^\perp$.
Since $z=0$ is an isolated and non-degenerate eigenvalue of $L\big|_{\h_m}$, we conclude $\sigma\big(L\big|_{\h_m^\perp}\big) = \sigma\big(L\big|_{\h_m}\big)\setminus \{0\}$.
Since $L$ generates a $C_0$-semigroup on $\h_m^\perp$, it is closed and its resolvent is analytic on
$$\varrho\big(L  \big|_{\h_m^\perp}\big) \supseteq \{ z \in \C: \text{\rm Re} \, z > - \gamma_m \}.$$
For any fixed $\widetilde \gamma_m \in (0,\gamma_m)$ we
henceforth conclude from \eqref{estim_line} that the resolvent of $L \big|_{\h_m^\perp}$ is  uniformly bounded (on a whole right half space)
\begin{equation} \label{resolvent-uniform}
 \left\| (L-z)^{-1} \right\|_{\mathscr B(\h_m^\perp)} \le M(\widetilde \gamma_m)<\infty,
 \quad \text{\rm Re} \, z \ge -\widetilde \gamma_m\,.
\end{equation}
Note, however, that the constant $M(\widetilde \gamma_m)$ is not known explicitly.\\

\emph{Step 3:}
Next we shall show that this resolvent estimate yields an exponential decay estimate for the semigroup $U_t$ on $\h_m^\perp$. In order to do so, we will
apply the Gearhart-Pr\"uss-Theorem to the rescaled semigroup $\E^{\widetilde \gamma_m t}U_t$, cf.\ Theorem V.1.11 in \cite{EN} (see also \cite[Theorem 3.1]{GMM} and \cite{HeSj}). This
is possible in view of
the uniform bound \eqref{resolvent-uniform} and yields the following estimate for $U_t$:
\begin{align}\label{U-decay}
\| U_t\|_{\mathscr B(\h_m^\perp)} \le \frac{(1+M\omega)^2 C_L^2}{2\pi t} \,\E^{ - \widetilde\gamma_m t },
 \quad t>0,
\end{align}
where $\omega>\beta_m+\widetilde\gamma_m+1$,
$C_L\le\pi (\omega-\beta_m-\widetilde\gamma_m)^{-1}$, and
\begin{equation}\label{M}
 M:=\sup_{s\in\R}\| (L-(-\widetilde\gamma_m+is))^{-1} \| _{\mathscr{B}(\h_m^\perp)}
 \le\vartheta_m(- \widetilde\gamma_m),
\end{equation}
where the second inequality follows directly from \eqref{estim_line} and the definition of $M$.
Note that \eqref{M} asserts a bound on the resolvent along the (fixed) line $$\{z\in \C:\ z=-\widetilde\gamma_m+is\},$$ in contrast to \eqref{resolvent-uniform}.
Keeping this in mind, we conclude that the estimates in the proof of Theorem V.1.11 in \cite{EN} in fact only depend on the resolvent evaluated at $ z=-\widetilde\gamma_m+is$.
Interpolating \eqref{U-decay} with (\ref{boundsemi}), we consequently conclude
$$
\| U_t  \|_{\mathscr B(\h_m^\perp)} \le \delta_m \E^{ - \widetilde\gamma_m t },
$$
where (after optimizing in $\omega>\beta_m+\widetilde\gamma_m+1$)
\begin{align}\label{C}
\delta_m := \max \left\{ \frac{\pi}{2}
\vartheta_m(- \widetilde\gamma_m)^2
 \, ; \, \E^{\beta_m +  \widetilde\gamma_m} \right\}>1.
\end{align}
This finishes the proof.

\end{proof}

Proposition \ref{gap_big_space} implies that the operator $L$ is
invertible in the space $\h_m^\perp$, for $m\ge Kd$. More precisely, invoking classical arguments (cf.\ Theorem 1.5.3 in \cite{Pa}), we infer
\begin{align}\label{sigmatilde}
 \forall \, m \ge Kd: \quad {\| L^{-1}\, \|}_{\mathscr{B}(\h_m^\perp)} \le
 \frac{\delta_m}{\widetilde \gamma_m}
 =: \frac{1}{\sigma_m}.
\end{align}
Note that the constant ${\sigma_m}>0$ depends
on $m$ and thus on $d$ and $A$. The reason why we obtain exponential decay of $U_t$ on $\h_m$ with a rate $\widetilde \gamma_m < \gamma_m$
can be understood from the fact that in order to apply the Gearhart-Pr\"uss Theorem one needs to guarantee
a uniform resolvent estimates on some complex half plane, which is not
sharp, in contrast to e.g. the Hille-Yoshida theorem (used on $\h$).


\section{Boundedness of the perturbation} \label{sbound}

In this section we shall prove the boundedness of the operator $\Theta [V_0]$ in $\h_m$ which is the key technical result for our perturbation analysis. We recall that our potential $V$ from \eqref{Potential} consists of the harmonic potential plus the perturbation $\lambda V_0$. Hence, we shall now consider $-\lambda \Theta[V_0]$ as a perturbation of $L$.

\begin{proposition} \label{lempert}
Let $m\in \N$ and the potential $V_0\in C^m_b(\R^d)$.
Then the operator $\Theta[V_0]$ maps $\h_m$ into ${\h}_m^\perp$ and
\[
\left\| \Theta[V_0] \right\|_{\mathscr B(\h_m )} \leq \Gamma_m:=C_m \, \max_{
 |j|\le m } \, \left\| \partial_{x}^j V_0 \right\|_{L^\infty(\R^d)},
\]
where $C_m>0$ denotes some positive constant, depending only on $m$ and $d$.
\end{proposition}

\begin{proof}
 We first note that, in view of \eqref{A-bound}, the norm $\|f\|^2_{\h_m}$ and the norm
$$
\| f \|_{m}^2:= \quad \iint_{\R^{2d}} |f|^2(x,\xi) \, \left( 1 + \left(|x|^2 +
|\xi|^2\right)^m \right) \; \D x \, \D \xi
$$
are equivalent.
It is therefore enough to prove that
\begin{align}\label{est}
 \iint_{\R^{2d}} |\Theta [V_0]w |^2 \, \left(1 +
   \left(|x|^2 + |\xi|^2\right)^m \right)\; \D
 x \, \D \xi \le C_0 \, \| w \|_{m}^2:
\end{align}
for some $C_0\ge 0$. In the following we denote by
\begin{equation}\label{partfourier}
 (\mathcal{F}_{\xi \to \eta} w)(x,\eta)\equiv {\widehat w} (x, \eta) := \int_{\R^d} w (x, \xi) \, \E^{- \I \xi \cdot \eta} \D \xi,
\end{equation}
the partial Fourier transform with respect to the variable $\xi \in \R^d$ only. Recall from \eqref{theta}
that the operator $\Theta[V_0]$ acts via
\begin{equation}\label{action}
 \Theta[V_0]w= -\I \mathcal{F}_{\eta \to \xi}^{-1} \big( \delta
   V_0(x,\eta) \cdot \mathcal{F}_{\xi \to \eta} w(x,\eta) \big).
\end{equation}
Using Plancherel's formula and H\"older's inequality, this implies
$$
\left\| \Theta[V_0]w \right\|_{L^2} \le 2\left\| V_0
\right\|_{L^\infty} \, \| w \|_{L^2}.
$$
Thus, in order to prove \eqref{est}, we only need to estimate
$$
\iint_{\R^{2d}} | \Theta [V_0]w(x,\xi) |^2 \, \left(|x|^2
  + |\xi|^2\right)^m \; \D x \, \D \xi .
$$
We rewrite this term, using
$$
\left(|x|^2 + |\xi|^2\right)^m = \sum_{j=0}^m \binom{m}{j} \,
|x|^{2(m-j)} \, |\xi|^{2j},
$$
in the following form:
\begin{align*}
 & \iint_{\R^{2d}} |\Theta [V_0]w(x,\xi)|^2 \, \left(|x|^2
   + |\xi|^2 \right)^m\; \D x \, \D \xi  \\
   & = \sum_{j=0}^m \binom{m}{j} \, \iint_{\R^{2d}} |x|^{2(m-j)} \,
   |\xi|^{2j} \, |\Theta [V_0]w(x,\xi)|^2\; \D x \, \D \xi.
\end{align*}
It holds
$$
|\xi|^{2j} = \sum_{|n |=j} c_{n,j} \, 
 \xi^{2 n_1} \cdot\cdot\cdot \xi^{2 n_d},
$$
where $c_{n,j}$ are some coefficients depending only on $n
\in \N^d$. Therefore
\begin{align*}
 &  \iint_{\R^{2d}} | \Theta [V_0]w(x,\xi)|^2 \,
 \left(|x|^2 + |\xi|^2 \right)^m\; \D x \, \D \xi  \\
 & = \ \sum_{j=0}^m \binom{m}{j}  \left( \sum_{|n|=j}
 c_{n,j} \, \int_{\R^{2d}} |x|^{2(m-j)} \, \xi^{2n} \, |\Theta
 [V_0]w(x,\xi)|^2\; \D x \, \D \xi \right),
\end{align*}
where we denote $\xi^{2n}:=\xi^{2 n_1} \cdot\cdot\cdot \xi^{2 n_d}$. From \eqref{action} we see that
\begin{align}\label{es2}
 \left\| \xi^{n} \, \left(\Theta [V_0]w \right)
 \right\|^2_{L^2(\R^{2d})} = (2\pi)^{-d}\left\| \partial_\eta^n \left(\delta
     V_0 \widehat{w}\right) \right\|^2_{L^2(\R^{2d})},
\end{align}
with $\widehat w(x,\eta)$ defined by \eqref{partfourier}. We expand
the right hand side of this identity by using the Leibniz formula (see also
\cite{MaBa}), and we apply
$$
\sup_{x,\eta \in \R^d} \left| \partial_{\eta_k}^j \left(\delta V_0
 \right) (x,\eta) \right| \le 2^{(1-j)} \, \sup_{y \in \R^d}
\left|\partial_{y_k}^j \, V_0(y) \right|,
$$
(cf.\ Definition \eqref{delta}). Then we can estimate (\ref{es2}) as follows:
\begin{align*}
\iint_{\R^{2d}} \xi^{2 n}  |\Theta [V_0]w (x,\xi)|^2 \; \D x\, \D \xi  \le C(n)  \max_{ |k|\le |n| }  \| \partial_{x}^j V_0 \|_{L^\infty(\R^d)}^2  \| \xi^{n-k}  w \|^2_{ L^2(\R^{2d})} ,
\end{align*}
where $C(n )>0$ depends only on binomial coefficients.  In summary, we obtain
\begin{align*}
&  \iint_{\R^{2d}}| \Theta [V_0]w(x,\xi)|^2 (|x|^2 + |\xi|^2)^m\; \D x \, \D \xi  \\
& \le \  \tilde C_m \sum_{j=0}^m \binom{m}{j}
  \Big(   \sum_{|n |=j}  c_{n,j} \max_{ |k|\le |n|} \| \partial_{x}^k V_0 \|_{L^\infty(\R^d)}^2
  \left\| | x|^{m-j} \xi^{n-k}  w\right\|^2_{L^2(\R^{2d})}\Big) \\
 & \le C_m^2\max_{ |j|\le m} \| \partial_{x}^j V_0\|_{L^\infty(\R^d)}^2    \iint_{\R^{2d}}  |w|^2(x,\xi) ( |x|^2 + |\xi|^2)^m \; \D x \, \D \xi,
 \end{align*}
with $\tilde C_m,\,C_m>0$ depending only on binomial coefficients. Thus, the assertion is proved.
\end{proof}

\begin{remark}\label{unbounded} The unboundedness of  $\Theta[V_0]$  in $\h$ (the exponentially weighted Hilbert space)
is due to its non-locality, which can be seen from the following reformulation of \eqref{theta} (cf.\ \S3 of \cite{ALMS}):
$$
 (\Theta [V] f)(x,\xi) =-2\,\pi^{-d}\,W(x,\xi)*_\xi f(x,\xi),
$$
with
$$
 W(x,k):=\mbox{ Im }\big[e^{2\I x\cdot k}\hat V(2k)\big].
$$
To illustrate the situation let us take $V_0(x)=\sin(x\cdot k_0)$ for some $k_0\in\R^d\setminus \{0\}$. This implies
$$
 (\Theta [V] f)(x,\xi) = 2^d \cos(2x\cdot k_0)\,[f(x,\xi-k_0)-f(x,\xi+k_0)].
$$
The problem is that such a shift operator cannot be bounded in an $L^2$-space with an inverse Gaussian weight, as the following computation shows:
$$
 \int_{\R^d} |f(\xi-k_0)|^2\,e^{|\xi|^2}\,\D\xi=
 e^{|k_0|^2}\int_{\R^d} |f(\xi)|^2\,e^{|\xi|^2}\,e^{2k_0\cdot \xi}\,\D\xi,\quad
 f\in L^2(\R^d,e^{|\xi|^2}\,\D\xi).
$$
\end{remark}


\section{Proof of the main theorems} \label{sproof}

The results of the preceding sections allow us to give the proofs of  Theorem \ref{th1} and Theorem \ref{th2}.

\begin{proof}[Proof of Theorem \ref{th1}] We start with Assertion (i): Let $m$ be the integer fixed in the assertion.
 Any solution $w_\infty\in \h_m$ of \eqref{qfpS1} that is subject to
 the normalization $\iint w_\infty\,\D x\D \xi=1$, satisfies
 the unique decomposition $w_\infty=\mu+w_*$ with $w_*\in
 \h_m^\perp$, i.e. $\iint w_* \D x \D \xi = 0$. Therefore, we consider the following fixed point
 iteration for $w_*$:
\begin{align*}
 T: \h_m^\perp \rightarrow \h_m^\perp , \quad w_{n-1} \mapsto
 T(w_{n-1})\equiv w_n,
\end{align*}
where $w_n\in \h_m^\perp$ solves
\begin{align*}
L w_n = \lambda \, \Theta\left[V_0\right]\left(w_{n-1} + \mu\right).
\end{align*}
To be able to apply Banach's fixed point theorem, we have to prove
that the mapping $T$ is a contraction on $\h_m^\perp$. To this end we
write, for any $w_{n-1}, \widetilde w_{n-1} \in \h_m^\perp$,
\begin{align*} {\| w_n - \widetilde w_n \|}_{\h_m^\perp} = \, \left\| \,
   \lambda \, L^{-1} \, \Theta\left[V_0\right]\left(w_{n-1} -
     \widetilde w_{n-1} \right) \,
   \right\|_{\h_m^\perp }
\end{align*}
and estimate
\begin{align*} {\left\| w_n - \widetilde w_n \right\|}_{\h_m^\perp} \leq \,
 |\lambda| \, {\left\| \, L^{-1} \, \right\|}_{\mathscr B(\h_m^\perp)} \, {\left\|
   \Theta[V_0] (w_{n-1} - \widetilde w_{n-1}) \, \right\|}_{\h_m^\perp }\,
 .
\end{align*}
{}From \eqref{sigmatilde} and Proposition \ref{lempert} we obtain
\begin{align*} {\| w_n - \widetilde w_n \|}_{\h_m^\perp} \leq \,
 \frac{ \Gamma_m |\lambda|}{\sigma_m} \, \left\| w_{n-1} - \widetilde
   w_{n-1} \right\|_{\h_m^\perp},
\end{align*}
since the potential $V_0$ satisfies \eqref{condV}. Since
$|\lambda| < \sigma_m / \Gamma_m$, there exists a unique fixed point
$w_*= T(w_*)\in \h_m^\perp$.  Thus, the unique (stationary) solution
of \eqref{qfpS1} is obtained as $w_\infty = \mu + w_* \in \h_m$. Note, however, that $\mu \not \perp w_*$ in the sense of $\h_m$.

The obtained solution $w_\infty$ is real valued, since $T$ maps real valued functions to real valued functions. Moreover, $w_\infty\in\h_m$ satisfies
\eqref{qfpS1}, at least in the distributional sense. Furthermore,
$\Theta[V_0]w_\infty \in\h_m$ and $Lw_\infty \in H^{-2}(\R^{2d})$ and
thus \eqref{qfpS1} also holds in $H^{-2}(\R^{2d})$.  To explore --
{\it a posteriori} -- the regularity of $w_\infty$, we rewrite \eqref{qfpS1}
in the following weak form
$$
 \iint_{\R^{2d}}(\nabla_x w_\infty \cdot \nabla_x\varphi + \nabla_\xi w_\infty \cdot \nabla_\xi\varphi +w_\infty
 \varphi) \, \D x\,\D\xi\,=\,{}_{H^{-1}}\langle F( w_\infty) ,\varphi \rangle_{H^1},
$$
for any $\varphi\in H^1(\R^{2d})$, where
$$
F (w_\infty):= w_\infty- \text{div}_x(\xi w_\infty)+\text{div}_\xi(x w_\infty +2\xi w_\infty )-\lambda\Theta[V_0]w_\infty .$$
Clearly $F(w_\infty) \in H^{-1}(\R^{2d})$ and thus $w_\infty \in H^{1}(\R^{2d})$ follows. Moreover, since $F(w_\infty )\in L^2_{loc}(\R^{2d})$,
we also have $w_\infty \in H^{2}_{\rm loc}(\R^{2d})$.\\


For the proof of Assertion (ii), we first note that $\Theta[V_0]$ is
a bounded perturbation of $L$ on $\h_m$ and thus \eqref{qfp} admits a unique mild solution $w \in C([0,\infty), \h_m)$.
Since $\Theta[V_0]$ maps $\h_m$ into $\h_m^\perp$, we also know that along this solution the mass is conserved, i.e. $\iint w(t) \, \D x \D \xi = 1$, for all $t \geq 0$.

Next, consider the new unknown $g(t):= w(t) - w_\infty$ with $g_0=w_0-w_\infty$. Due to mass conservation $g(t) \in {\h}_m^\perp$ for all $t \geq 0$, and we also have
$$
\partial_t g = L g - \lambda \Theta[V_0]  g,
$$
since $w_\infty$ is a stationary solution of \eqref{qfp}.
Taking into account that the semigroup $U_t$ associated with $L$ in the space $\h_m$ satisfies (\ref{semigr2}), it holds:
\begin{align*}
\| g(t) \|_{\h_m} & \le \delta_m \E^{-\widetilde\gamma_m t} \| g_0\|_{\h_m} + \delta_m |\lambda| \int_{0}^t \E^{-\widetilde\gamma_m(t-s)} \| \Theta[V_0]g(s)\|_{\h_m}\; \D s \\
& \le \delta_m \E^{-\widetilde\gamma_m t} \| g_0\|_{\h_m} + \delta_m |\lambda| \Gamma_m \int_{0}^t \E^{-\widetilde\gamma_m(t-s)} \| g(s) \|_{\h_m}\; \D s,
\end{align*}
for any $\widetilde\gamma_m\in(\delta_m |\lambda| \Gamma_m, \gamma_m)$
and with $\Gamma_m$ defined in Proposition \ref{lempert}.
Gronwall's lemma then implies
\begin{align*}
  \| g(t) \|_{\h_m} \le \delta_m \E^{-t (\widetilde\gamma_m-\delta_m|\lambda| \Gamma_m) } \|
  g_0\|_{\h_m}.\\
\end{align*}

It remains to prove Assertion (iii):
As before we write $w_\infty = \mu + w_*$, where $w_*\in\h_m^\perp$ solves
$$
\left(L - \lambda \Theta[V_0]\right)w_* = \lambda \Theta[V_0] \mu.
$$
Since ${\|\Theta[V_0]  \|}_{\mathscr B(\h_m)} \leq \Gamma_m$, we obtain ${\|\lambda \Theta[V_0] \mu \|}_{\h_m} \leq  |\lambda| \Gamma_m \|\mu\|_{\h_m}$. Next, we consider $L$ on $\h_m^\perp$.
{}From the proof of Proposition \ref{gap_big_space} 
we conclude for its resolvent set:
$$
\varrho\left(L\big|_{{\h}_m^\perp}\right)\supseteq
\left\{ z \in \C: \text{Re}\, z > - \gamma_m  \right\},
$$
and thus
$$
\varrho \left( (L - \lambda \Theta[V_0])\Big|_{{\h}_m^\perp} \right)
\supseteq \left\{ z \in \C: \text{Re} \, z  > |\lambda| \Gamma_m -  \gamma_m
\right\}.
$$
Since $|\lambda|\Gamma_m-\gamma_m<0$ (see \eqref{cond-lambda}) we have
$$
\left(L - \lambda \Theta[V_0]\right)^{-1} = L^{-1} \left( {\rm Id} - \lambda
\Theta[V_0] L^{-1}\right) ^{-1} \qquad (\mbox{on }\h_m^\perp).
$$
Using \eqref{sigmatilde} and $|\lambda|\Gamma_m<\sigma_m$ (see \eqref{cond-lambda}) we conclude
$$
\big\| {(L - \lambda \Theta[V_0])} \Big |_{\h_m^\perp} ^{-1}\ \big\|_{\mathscr B({{\h}}_m^\perp )} \leq \frac{1}{\sigma_m} \, \frac{ 1}{ 1 - |\lambda| \Gamma_m \frac{1}{\sigma_m} } = \frac{1}{ \sigma_m - |\lambda| \Gamma_m}.
$$
Thus, by writing
$$
w_* = \left( (L - \lambda \Theta[V_0])\Big|_{\h_m^\perp}\right)^{-1}
(\lambda \Theta[V_0] \mu),
$$
we infer
$${\| w_* \|}_{\h_m} \equiv {\| w_\infty - \mu \|}_{\h_m} \leq\frac{ |\lambda| \Gamma_m\|\mu\|_{\h_m}}{ \sigma_m - |\lambda| \Gamma_m}$$
and the assertion is proved.
\end{proof}

\begin{remark}
Due to the mass normalization $\iint w_\infty\,\D x\D \xi =\iint \mu\,\D x\D \xi =1$,
the fixed point $w_*$ must take both positive and negative values. Thus, $w_\infty=\mu+w_*$ may,
in general, also take negative values.
\end{remark}

\begin{proof}[Proof of Theorem \ref{th2}] We start with assertion (ii), which follows from the fact that
\begin{align*}
{\| \, \rho \, \|}_{\mathscr T_2} = {\| \, \rho (\cdot, \cdot) \, \|}_{L^2} =
(2\pi)^{d/2} \, {\| \, w (\cdot, \cdot) \, \|}_{L^2} \leq (2\pi)^{d/2} \, \, {\| \,
 w(\cdot, \cdot) \|}_{\h_m}, \quad \forall \, \rho \in \mathscr T_2.
\end{align*}
Thus, we infer
$$
\rho(t) \stackrel{t \rightarrow
\infty}{\longrightarrow}  \rho_\infty \quad \mbox{ in $\mathscr T_2$ },
$$
with the exponential rate obtained from Theorem 1 (ii).

To prove assertion (i) we consider the transient equation \eqref{qfp} as an auxiliary problem: Choose
any $\rho_0 \in \mathscr T_1^+$ such that $\tr \rho_0 = 1$ and the corresponding $w_0 \in \h_m$. Due to the
results on the linear Cauchy problem given in \cite{ArSp} we  know that \eqref{qfp} gives rise to
a unique mild solution $\rho \in C([0,\infty); \mathscr T^+_1)$, satisfying $\tr \rho(t)=1$, for all $t\geq 0$.
Hence, the trajectory $\{\rho(t),\,t\ge0\}$ is bounded in $\mathscr T_1$. Since $\mathscr T_1$ has a predual, i.e.
the compact operators on $L^2(\R^d)$,
the Banach-Alaoglu Theorem then asserts the existence of a
sequence $\{ t_n \}_{n \in \N} \subset \R_+$ with $t_n\to\infty$, such that
$$
\rho(t_n) \stackrel{n \rightarrow
\infty}{\longrightarrow} \widetilde \rho \quad \mbox{ in $\mathscr T_1$ weak-$\star $}\,
$$
for some limiting $\widetilde \rho \in \mathscr T_1$.
The already obtained $\mathscr T_2$-convergence of $\rho(t)$ towards $\rho_\infty \in \mathscr T_2$ implies $\rho_\infty = \widetilde \rho \in \mathscr T_1$. And the
\emph{uniqueness} of the steady state yields the convergences of the whole $t$-dependent function $\rho(t) \to \widetilde \rho$ in $\mathscr T_1$ weak-$\star $. Finally, we also conclude positivity of the operator $\rho_\infty $ by the
$\mathscr T_2$-convergence and the fact that we already know from \cite{ArSp}: $\rho(t) \geq 0$, for all $t\geq 0$.

It remains to
prove $\tr \rho_\infty = 1$. To this end, we recall that for any $\rho \in \mathscr T_1^+$ the corresponding kernel
$$
\vartheta (x, \eta) := \rho \left( x+\frac{\eta }{2}\ , \ x-\frac{\eta }{2} \right)
$$
satisfies $ \vartheta \in C(\R^d_\eta, L^1_+(\R^d_x))$, see \cite{Ar}, and it also holds
\begin{equation}\label{formula}
\tr \rho = \int_{\R^d} \vartheta (x, 0) \, \D x.
\end{equation}
Further, note that $\vartheta (x,\eta) = (\mathcal F_{\xi \to \eta}
w)(x,\eta) \equiv \hat w(x,\eta)$, by \eqref{trans}.  On the other
hand, for any $w \in \h_m$ we know that $ \hat w \in C(\R^d_\eta,
L^1(\R^d_x))$, due to the polynomial $L^2$-weight $\nu_m^{-1}$ in $x\in
\R^d$ and a simple Sobolev imbedding w.r.t.\ the variable $\eta \in
\R^d$ (for both embeddings we used $m>\frac{d}{2}$). Hence the normalization condition $\iint w_\infty \, \D x \D
\xi = 1$ implies $\tr \rho_\infty = 1$, via \eqref{formula}, and
assertion (ii) is proved.

Finally, we prove claim (iii) by first noting that the $\mathscr
T_2$-convergence of $\rho(t)$ implies convergence in the strong
operator topology. Thus, having in mind that ${ \| \rho(t) \|
}_{\mathscr T_1} = { \| \rho_0 \| }_{\mathscr T_1} = 1$, we infer from
Gr\"umm's theorem (Th.\ 2.19 in \cite{Si}) that $\rho(t)$ also converges in the
$\mathscr T_1$-norm towards $\rho_\infty$. This concludes the proof of
Theorem \ref{th2}.
\end{proof}

\appendix

\section{Proof of Proposition \ref{prop_assum_AB}}\label{S7.1}

The proof will be divided into several steps:\\

\emph{Step 1:} Let $\chi\in C_0^\infty(\R^{2d})$ be such that $\chi=1$ on
 $B_1(0)$, with $\supp(\chi)\subseteq B_2(0)$,
 $\|\nabla \chi\|_{L^\infty}\le \sqrt2$, and let $\chi_\varepsilon(y) :=
 \chi(\varepsilon \, y)$, for any $y=(x,\xi) \in \R^{2d}$, $0<\varepsilon<1$.
 We define
\begin{align}\label{L1varepsilon}
L_1^\varepsilon w :=  \big(d \, w - \nu_m
\, \nabla \nu_m^{-1} \cdot \nabla w \big) \, \chi_\varepsilon,
\end{align}
as well as
\begin{align}\label{L2varepsilon}
L_2^\varepsilon w := & \ \nu_m \, \diver \left( \nu_m^{-1} \, \nabla
 w \right) + \nabla w \cdot \nabla A + d \, w +
\diver(w \, F)  \\
& +  \left(d \, w -
\nu_m \, \nabla \nu_m^{-1} \cdot \nabla w \right)\,\left(1-\chi_\varepsilon\right).\nonumber
\end{align}
It is easily seen that $L_1^\varepsilon$ indeed satisfies property (1).\\

\emph{Step 2:} In order to prove properties (2) and (3), we need to show that
$$
( L_2^\varepsilon-z)^{-1}:  {\mathcal H}_m \to {\mathcal {H}}_m^1,
$$
is bounded for $z \in \Omega \subset \C$.
To this end, it suffices to show that $ (L_2^\varepsilon-z)$ satisfies $\forall\,z\in\Omega$:
\begin{align}\label{coercive}
- {\rm Re}\,\langle  (L_2^\varepsilon-z) w, w \rangle_{{ \mathcal H}_m} \ge  c \|  w \|_{{\mathcal H}_m^1}^2,
\end{align}
with some $c=c({\rm Re}\,z)>0$.
Indeed, suppose (\ref{coercive}) holds
for all (complex valued) $w\in {\mathcal H}_m^2:=\{w\in {\mathcal H}_m : \nabla w, (\nabla A+F)\cdot \nabla w, \Delta w \in {\mathcal H}_m\}$.
Then, $L_2^\varepsilon-z$ is densely defined on $\mathscr D(L_2^\varepsilon):={\mathcal H}_m^2\subset \mathcal H_m$ and dissipative. Hence, it has a maximal dissipative (and thus surjective on ${\mathcal H}_m$) extension, which we shall consider in the sequel.
{} From (\ref{coercive}) we conclude
\begin{align*}
{c} \| w \|^2_{{\mathcal H}_m^1}\le   - {\rm Re}\,\langle ( L_2^\varepsilon -z) w, w \rangle_{ {\mathcal H}_m} \le  \| w \|_{{ \mathcal H}_m^1} \| ( L_2^\varepsilon -z) w \|_{{\mathcal H}_m}, \quad {c}>0,
\end{align*}
which implies
\begin{align}\label{coercive3}
\| ( L_2^\varepsilon -z)^{-1} \tilde w \|_{{\mathcal H}_m^1}\le
\frac{1}{{c}} \, \| \tilde w \|_{{\mathcal H}_m},\quad
\textrm{for all} \;  \tilde w \in {\mathcal H}_m.
\end{align}
For future reference we note that this bound will turn out to be uniform on the lines $z=a+is, \,s\in\R$ (with fixed $a>-\Lambda_m$), since $c=c({\rm Re}\,z)$.\\

\emph{Step 3:} Now, we have to prove \eqref{coercive}. To this end, we decompose
\begin{align*}
 \text{Re}\, \left\langle L_2^\varepsilon w, w
 \right\rangle_{{\mathcal{H}}_m} =
 &-\iint_{\R^{2d}} \left(1+A^m\right) |\nabla w|^2 \; \D x \, \D \xi \\
 &- \frac{m}{2} \, \iint_{\R^{2d}} |w|^2 \, A^{m-1} \, \left|\nabla
   A\right|^2 \; \D x \, \D \xi \\
   & +d \, \iint_{\R^{2d}} |w|^2 \,
 \left(1+A^m\right) \,
 \left(1-\chi_\varepsilon \right) \; \D x \, \D \xi\\
 &+ \frac{1}{2} \, \iint_{\R^{2d}} |w|^2 \, \nabla
 \left(1-\chi_\varepsilon\right)\cdot \nabla \left(1+A^m\right)
 \; \D x \, \D \xi \\
 & + \frac{1}{2} \, \iint_{\R^{2d}} |w|^2 \,
 \left(1-\chi_\varepsilon\right) \, \Delta \left(1+A^m\right)
 \; \D x \, \D \xi\\
 =:& \ I_1 + I_2 + I_3 + I_4 + I_5.
\end{align*}
Using Lemma \ref{techlemma} (b), we readily obtain for $\varepsilon\le\frac1{\sqrt3}$:
\begin{align*}
{\left|I_3\right|}  & \ \le \frac{m}{4} \,
\iint_{\R^{2d}} |w|^2 \, A^{m-1} \, \left|\nabla A \right|^2 \,
\left(1-\chi_\varepsilon \right)\; \D x \, \D \xi \\
&\  \le \frac{m}{4} \, \iint_{\R^{2d}} |w|^2 \, A^{m-1} \, \left|\nabla
 A\right|^2  \; \D x \, \D \xi.
\end{align*}
In order to treat the term $I_4$, we note that
\[
| \nabla \chi_\varepsilon | =
\varepsilon | \nabla \chi |, \ |\nabla \chi| \le \sqrt2, \  \textrm{supp}\; \{
\nabla \chi_\varepsilon\} \subset \Big \{ y\in \R^{2d}\; \big| \;
\frac{1}{\varepsilon} \le |y| \le \frac{2}{\varepsilon} \Big \}.
\]
Therefore, by \eqref{gradA-bound},
\[
\frac{1}{ |\nabla A |} \le \sqrt{18}\, {\varepsilon} \quad
\text{for} \quad  y=(x,\xi) \in
\textrm{supp}\; \{ \nabla \chi_\varepsilon\}.
\]
With $\e < 1$ this allows us to estimate
\begin{align*}
 \left| \nabla \left(1-\chi_\varepsilon\right)\cdot \nabla
   \left(1+A^m\right) \right| & \le \varepsilon \, m \, A^{m-1} \,
 \left|\nabla \chi \right| \, \left|\nabla A \right| \\
  & \le  6\varepsilon^2 \, m  \, A^{m-1} \, \left| \nabla A \right|^2  \le 6{\varepsilon} \, m \, A^{m-1} \, \left| \nabla A \right|^2,
\end{align*}
which implies
$$
| I_4 | \le 3m\, \varepsilon \, \iint_{\R^{2d}} |w|^2 \, A^{m-1} \,
\left|\nabla A\right|^2 \; \D x \, \D \xi.
$$
The term $I_5$ can be estimated using Lemma \ref{techlemma} (d):
\begin{align*}
 I_5 & = \frac{1}{2} \, \iint_{\R^{2d}} |w|^2 \,
 \left(1-\chi_\varepsilon\right) \, \Delta (1+A^m)  \; \D x \, \D \xi \\
 & \le 3{m} \varepsilon^2 \, {(m-1+3d)} \,
 \iint_{\R^{2d}} |w|^2 \, \left(1-\chi_\varepsilon\right) \, A^{m-1} \,
 \left|\nabla A\right|^2  \; \D x \, \D \xi\\
 & \le 3{m} \, \varepsilon^2 \, {(m-1+3d)} \,
 \iint_{\R^{2d}} |w|^2  \, A^{m-1} \, \left|\nabla
   A\right|^2 \; \D x \, \D \xi.
\end{align*}
In summary, we obtain
\begin{align*}
 \frac{1}{4} \, I_2 + |I_4 |\le -\left( \frac{m}{8} -
   3m\varepsilon \right) \, \iint_{\R^{2d}} |w|^2 \, A^{m-1}
 \, \left|\nabla A\right|^2 \; \D x \, \D \xi,
\end{align*}
and
\begin{align*}
 \frac{3}{4} \, I_2 + I_3 +I_5 \le - \left( \frac{m}{8} -
   {3m\varepsilon^2} \, {(m-1+3d)} \right) \,
 \iint_{\R^{2d}}|w|^2 \, A^{m-1} \, \left|\nabla A\right|^2
 \; \D x \, \D \xi.
\end{align*}
Now choosing $\varepsilon \le \min\left\{\frac1{24};\frac{1}{12\sqrt{m}}\right\}$,
we can estimate (using $m\ge d$)
\begin{align*}
 m \, \left( \frac{1}{8} - 3\varepsilon \right)  \ge 0,\quad
 m \, \left( \frac{1}{8} - 3{\varepsilon^2}(m-1+3d)\right)  \ge
 \frac{m}{24}.
\end{align*}
Therefore
\begin{align*}
-\text{Re}\, \langle  L_2^\varepsilon w, w \rangle_{{\mathcal{H}}_m} \ge & \ \iint_{\R^{2d}} (1+A^m) |\nabla w|^2  \D x \, \D \xi  \\
& + \left( \frac{m}{8}-2 \right) \iint_{\R^{2d}} |w|^2A^{m-1}|\nabla A|^2  \D x \, \D \xi,
\end{align*}
which we estimate further using Lemma \ref{techlemma} (b):
\begin{align*}
-\text{Re}\, \langle  L_2^\varepsilon w, w \rangle_{{\mathcal{H}}_m} \ge & \iint_{\R^{2d}} (1+A^m) |\nabla w|^2 \; \D x \, \D \xi  \\
& +  \frac{d}{6} \iint_{\R^{2d}} |w|^2(1+A^{m}) (1-\chi_{1/\sqrt{12}})\; \D x \, \D \xi.
\end{align*}
This establishes the desired estimate outside of $(x,\xi)\in B_{1/\sqrt{12}}(0)$.
In order to take into account the contribution near $|y|=0$, we consider
$$
\iint_{\R^{2d}} (1+A^m) \, \left|\nabla w\right|^2 \; \D x \, \D \xi .
$$
Applying Sobolev's inequality we obtain for any $d> 1$:
\begin{align*}
 \iint_{\R^{2d}} \left(1+A^m\right) \, \left|\nabla w\right|^2 \; \D
 x \, \D \xi & \ge \iint_{\R^{2d}} \left|\nabla w\right|^2 \; \D x
 \, \D \xi  \ge
  C_d^2 \, \|w \|^2_{L^{q}(\R^{2d})} ,
\end{align*}
with $ q=\frac{4d}{2d-2}$. This can be estimated further via
\begin{align*}
C_d^2 \, \|w \|^2_{L^{q}(\R^{2d})}  \ge    C_d^2 \, \|w \|^2_{L^{q}(B_{1/\sqrt{12}}(0))} \ge
 \frac{ C_1^2}{ \|  1 + A^m \| _{L^\infty( B_{1/\sqrt{12}}(0))}}  \|w
 \|^2_{\mathcal{H}_{m}(B_{1/\sqrt{12}}(0))},
\end{align*}
where $C_1$ depends only on the Sobolev constant $C_d$ and on the
measure of the ball $B_{1/\sqrt{12}}(0)$.
Finally, in order to deal with $d=1$ we apply Cauchy-Schwarz to obtain
\begin{align*}
 \iint_{\R^{2}} | \nabla w| \; \D x \, \D \xi &= \iint_{\R^{2}} |
 \nabla w| \,
 \left(1+A^m\right)^{1/2} \, \left(1+A^m\right)^{-1/2} \D x \, \D \xi \\
 &\le \left( \iint_{\R^{2}} \, \left|\nabla w\right|^2 \,
   \left(1+A^m\right)  \D x  \D \xi \right)^{1/2} \,
\left( \iint_{\R^{2}} \left(1+A^m\right)^{-1}   \D x  \D \xi \right)^{1/2}.
 \end{align*}
By assumption we have $m> 1$ (cf.\ Remark \ref{Rem4.4}).
Hence, the second factor on the r.h.s.\ is a finite constant, denoted by $C_{A,m}$.
Applying again Sobolev's inequality (for $d=1$), we obtain
\begin{align*}
 \| \nabla w \|_{L^1( \R^{2})} \ge \widetilde C_1 \, \| w \|_{L^2( \R^{2})} \ge \frac{ \widetilde C_1}{\| 1 + A^m \|^{1/2} _{L^\infty( B_{1/\sqrt{12}}(0))}}
 \, \| w \|_{\mathcal{H}_m(B_{1/\sqrt{12}}(0))} .
\end{align*}

By combining all the above estimates we infer for all (complex valued) $w\in {\mathcal H}_m^2$:
\begin{align*}
 -\text{Re}\,  \left\langle L_2^\varepsilon w, w \right\rangle_{{\mathcal{H}_m}}
  \ge \Lambda_m \, \left\langle w, w \right\rangle_{{\mathcal{H}^1_m}} ,
\end{align*}
where $\Lambda_m >0$ is given by
\begin{equation}\label{Lambda_m}
 \Lambda_m := \textrm{min} \Big \{ \frac{C_1^2}{{2 \| 1 + A^m \|_{L^\infty( B_{1/\sqrt{12}}(0))}} }\, ;\, \frac{\widetilde C_1^2}{ 2C_{A,m}^2 \| 1 + A^m \|_{L^\infty( B_{1/\sqrt{12}}(0))}} \, ; \,
   \frac{1}{2}\, ;\, \frac{d}{6} \Big \}.
\end{equation}
In summary, inequality (\ref{coercive}) holds on ${{\mathcal H}}_m^2$ for any ${\rm Re} \, z > - \Lambda_m$.
And we easily see that $c(z)=\min\{\Lambda_m\,;\,\Lambda_m+{\rm Re}\,z\}$.
Hence, the operator $(L_2^\varepsilon -z)$ is invertible in ${{\mathcal H}}_m^1$. And we have to choose $\varepsilon = \varepsilon(m) >0$ such that
\begin{align}\label{eps}
 \varepsilon \le \textrm{min}\left\{ \frac{1}{12\sqrt{m}}\, ; \, \frac{1}{24} \right\}.
\end{align}
\hfill $\Box$

\bibliographystyle{amsplain}

\end{document}